\documentclass[a4paper, 11pt]{article}
\usepackage{mathrsfs}
\usepackage{mathrsfs}
\usepackage{amsmath,amssymb}
\usepackage{amsfonts}
\usepackage[T1]{fontenc}
\usepackage[latin9]{inputenc}
\usepackage{amsthm}
\usepackage{graphicx}
\usepackage{amsmath}

\usepackage{amsthm}
\usepackage{array}
\usepackage{cases}
\usepackage{enumerate,enumitem,todonotes}
\usepackage{thmtools}
\usepackage{thm-restate}
\usepackage{setspace}  
\usepackage{changes}
\usepackage{float} 

\makeatletter
\def\th@plain{%
  \itshape 
}
\makeatother

\makeatletter
\renewenvironment{proof}[1][\proofname]{\par
  \pushQED{\qed}%
  \normalfont \topsep6\p@\@plus6\p@\relax
  \trivlist
  \item[\hskip\labelsep
        \bfseries
    #1\@addpunct{.}]\ignorespaces
}{%
  \popQED\endtrivlist\@endpefalse
}
\makeatother

\newtheorem{theorem}{Theorem}[section]

\numberwithin{equation}{section}

\newtheorem{thm}{Theorem}[section]
\newtheorem{cor}[thm]{Corollary}
\newtheorem{claim}[thm]{Claim}
\newtheorem{lemma}[thm]{Lemma}
\newtheorem{example}[thm]{Example}

\newtheorem{conj}[thm]{Conjecture}
\newtheorem{pblm}[thm]{Problem}

\numberwithin{equation}{section}

\usepackage[varg]{txfonts}
\usepackage{graphicx, epsfig, subfigure}
\usepackage{enumerate}
\usepackage[square, numbers, sort&compress]{natbib}
\ifx\pdfoutput\undefined
 \usepackage[dvipdfm,%
  pdfstartview=FitH,
  bookmarks=true,%
  bookmarksnumbered=true,
  bookmarksopen=true,
  plainpages=false,%
  pdfpagelabels,%
  colorlinks=true,
  linkcolor=blue,
  citecolor=blue,%
  urlcolor=black,
  pdfborder=001]{hyperref}
  \AtBeginDvi{}  
\else
 \usepackage[pdftex,%
  pdfstartview=FitH,
  bookmarks=true,%
  bookmarksnumbered=true,
  bookmarksopen=true,
  plainpages=false,%
  pdfpagelabels,%
  colorlinks=true,
  linkcolor=blue,
  citecolor=blue,%
  urlcolor=black,
  pdfborder=001]{hyperref}%
\fi

\usepackage{tikz}
\tikzset{ dot/.style={circle, draw=black, fill=black, inner sep=0pt, minimum size=5pt}}

\usepackage[left]{lineno}
\RequirePackage[normalem]{ulem} 
\RequirePackage{color}\definecolor{RED}{rgb}{1,0,0}\definecolor{BLUE}{rgb}{0,0,1} 
\providecommand{\DIFaddbegin}{} 
\providecommand{\DIFaddend}{} 
\providecommand{\DIFdelbegin}{} 
\providecommand{\DIFdelend}{} 
\providecommand{\DIFaddbeginFL}{} 
\providecommand{\DIFaddendFL}{} 
\providecommand{\DIFdelbeginFL}{} 
\providecommand{\DIFdelendFL}{} 
\newcommand{\DIFscaledelfig}{0.5}
\RequirePackage{settobox} 
\RequirePackage{letltxmacro} 
\newsavebox{\DIFdelgraphicsbox} 
\newlength{\DIFdelgraphicswidth} 
\newlength{\DIFdelgraphicsheight} 
\LetLtxMacro{\DIFOincludegraphics}{\includegraphics} 
\newcommand{\DIFaddincludegraphics}[2][]{{\color{blue}\fbox{\DIFOincludegraphics[#1]{#2}}}} 
\newcommand{\DIFdelincludegraphics}[2][]{
\sbox{\DIFdelgraphicsbox}{\DIFOincludegraphics[#1]{#2}}
\settoboxwidth{\DIFdelgraphicswidth}{\DIFdelgraphicsbox} 
\settoboxtotalheight{\DIFdelgraphicsheight}{\DIFdelgraphicsbox} 
\scalebox{\DIFscaledelfig}{
\parbox[b]{\DIFdelgraphicswidth}{\usebox{\DIFdelgraphicsbox}\\[-\baselineskip] \rule{\DIFdelgraphicswidth}{0em}}\llap{\resizebox{\DIFdelgraphicswidth}{\DIFdelgraphicsheight}{
\setlength{\unitlength}{\DIFdelgraphicswidth}
\begin{picture}(1,1)
\thicklines\linethickness{2pt} 
{\color[rgb]{1,0,0}\put(0,0){\framebox(1,1){}}}
{\color[rgb]{1,0,0}\put(0,0){\line( 1,1){1}}}
{\color[rgb]{1,0,0}\put(0,1){\line(1,-1){1}}}
\end{picture}
}\hspace*{3pt}}} 
} 
\LetLtxMacro{\DIFOaddbegin}{\DIFaddbegin} 
\LetLtxMacro{\DIFOaddend}{\DIFaddend} 
\LetLtxMacro{\DIFOdelbegin}{\DIFdelbegin} 
\LetLtxMacro{\DIFOdelend}{\DIFdelend} 
\DeclareRobustCommand{\DIFaddbegin}{\DIFOaddbegin \let\includegraphics\DIFaddincludegraphics} 
\DeclareRobustCommand{\DIFaddend}{\DIFOaddend \let\includegraphics\DIFOincludegraphics} 
\DeclareRobustCommand{\DIFdelbegin}{\DIFOdelbegin \let\includegraphics\DIFdelincludegraphics} 
\DeclareRobustCommand{\DIFdelend}{\DIFOaddend \let\includegraphics\DIFOincludegraphics} 
\LetLtxMacro{\DIFOaddbeginFL}{\DIFaddbeginFL} 
\LetLtxMacro{\DIFOaddendFL}{\DIFaddendFL} 
\LetLtxMacro{\DIFOdelbeginFL}{\DIFdelbeginFL} 
\LetLtxMacro{\DIFOdelendFL}{\DIFdelendFL} 
\DeclareRobustCommand{\DIFaddbeginFL}{\DIFOaddbeginFL \let\includegraphics\DIFaddincludegraphics} 
\DeclareRobustCommand{\DIFaddendFL}{\DIFOaddendFL \let\includegraphics\DIFOincludegraphics} 
\DeclareRobustCommand{\DIFdelbeginFL}{\DIFOdelbeginFL \let\includegraphics\DIFdelincludegraphics} 
\DeclareRobustCommand{\DIFdelendFL}{\DIFOaddendFL \let\includegraphics\DIFOincludegraphics} 
\title{\LARGE On $(1,2^4)$ and $(1,2^5)$-packing edge-coloring of sparse subcubic graphs}
\author{
Seog-Jin Kim\thanks{Department of Mathematics Education, Konkuk University, Korea. \texttt{skim12@konkuk.ac.kr}; This work was supported by the National Research Foundation of Korea (NRF) grant funded by the Korea government (MSIT) (RS-2026-25470022).}
\thanks{Korea Institute for Advanced Study (KIAS), Seoul, Korea.}
\and
Xujun Liu\thanks{Department of Applied Mathematics, Xi'an Jiaotong-Liverpool University, Suzhou, Jiangsu Province, 215123, China, \texttt{xujun.liu@xjtlu.edu.cn}; the research of X. Liu was supported by the National Natural Science Foundation of China under grants No.~12401466 and No. 12671413.}
\and
Boyan Xu\thanks{School of Data Science and Information Engineering, Guizhou Minzu University, Guiyang, Guizhou Province, 550025, China, \texttt{boyan04518@gmail.com}}
}
\date{}

\begin{document}

\maketitle

\begin{abstract}
A $(1^j,2^k)$-packing edge-coloring of a graph $G$ is a partition of the edge set $E(G)$ into $j$ matchings and $k$ induced matchings. Hocquard, Lajou, and Lu\v zar found a subcubic planar graph of girth $3$ that has no $(1,2^5)$-packing edge-coloring and also conjectured that every subcubic planar graph has a $(1,2^6)$-packing edge-coloring. We also notice that for every fixed positive integer $k$ there exists a subcubic planar graph with girth $k$ that is not $(1,2^3)$-packing edge-colorable. 

It is natural to consider what is the minimum positive integer $k_1$ such that every subcubic planar graph with girth at least $k_1$ is $(1,2^5)$-packing edge-colorable. Furthermore, we also consider what is the minimum positive integer $k_2$ such that every subcubic planar graph with girth at least $k_2$ is $(1,2^4)$-packing edge-colorable. In this paper, we show both $k_1$ and $k_2$ are finite, and in fact $5 \le k_1 \le 12$ and $6 \le k_2 \le 16$.

\end{abstract}

\section{Introduction}\label{sec:introduction}

For a graph $G$, let $L(G)$ denote its line graph.  The distance
between two distinct edges $e,f\in E(G)$, denoted by $d_G(e,f)$, is
the distance between their corresponding vertices of $L(G)$. Given a nondecreasing
sequence $S=(s_1,\ldots,s_t)$ of positive integers, an
$S$-packing edge-coloring of $G$ is a partition
\[
    E(G)=E_1\mathbin{\dot\cup}\cdots\mathbin{\dot\cup}E_t
\]
such that for every pair of distinct edges $e_1,e_2 \in E_i$, where
$i\in\{1,\ldots,t\}$, the distance between $e_1$ and $e_2$ is at least $s_i+1$. We use exponents to denote repeated entries in $S$, e.g., $(1,1,2,2,2)$ is denoted as $(1^2,2^3)$. In particular, a
$(1^j,2^k)$-packing edge-coloring is a partition of
$E(G)$ into $j$ matchings and $k$ induced matchings.

This framework interpolates naturally between proper and strong
edge-colorings.  Indeed, a $(1^j)$-packing edge-coloring is a proper
$j$-edge-coloring, while a $(2^k)$-packing edge-coloring is a strong
$k$-edge-coloring.  Recall that a graph is \emph{subcubic} if its
maximum degree is at most $3$, and that the \emph{girth} of a graph is
the length of a shortest cycle, with the girth of a forest taken to be
infinite.  By Vizing's theorem~\cite{V1}, every simple subcubic graph is
$(1^4)$-packing edge-colorable. Payan~\cite{P1}, and
independently Fouquet and Vanherpe~\cite{FV1}, proved the first
result in this flavor: every subcubic graph is
$(1^3,2)$-packing edge-colorable.

The notion strong $k$-edge-coloring, introduced by Fouquet and Jolivet~\cite{FJ1} in 1983, is equivalent to the $(2^k)-$packing edge-coloring. The strong chromatic index of a graph $G$, denoted by $\chi'_s(G)$, is the minimum $k$ such that $G$ has a strong $k$-edge-coloring. Erd\H{o}s and Ne\v set\v ril~\cite{EN1} conjectured that every graph $G$ with maximum degree $\Delta$ has $\chi'_s(G) \le \frac{5}{4} \Delta^2$ if $\Delta$ is even and $\chi'_s(G) \le \frac{5}{4} \Delta^2 - \frac{1}{2} \Delta + \frac{1}{4}$ if $\Delta$ is odd. Andersen~\cite{A1}, 
and independently, Hor\'ak, Qing, and Trotter~\cite{HQT1} proved that every
subcubic graph is $(2^{10})$-packing edge-colorable, which confirmed the conjecture of Erd\H{o}s and Ne\v set\v ril for the case when $\Delta = 3$. However, the conjecture remains open for the case when $\Delta \ge 4$. The best result for $\Delta = 4$ was proved by Huang, Santana, and Yu~\cite{HSY1}, who showed an upper bound of $21$.

The systematic study of $S$-packing edge-colorings was initiated by
Gastineau and Togni \cite{GT1}.  They proved that every cubic graph with a $2$-factor admits a
$(1^2,2^5)$-packing edge-coloring, and that four induced-matching
colors suffice when the graph is $3$-edge-colorable.  They also
exhibited a subcubic graph with no $(1,2^6)$-packing edge-coloring
and asked whether every cubic graph is
$(1,2^7)$-packing edge-colorable. 

Hocquard, Lajou, and Lu\v{z}ar~\cite{HLL1} substantially advanced the study of this topic. They proved that every subcubic graph
admits both a $(1,2^8)$-packing edge-coloring and a $(1^2,2^5)$-packing edge-coloring. For $3$-edge-colorable subcubic graphs, they
improved these bounds to $(1,2^7)$ and $(1^2,2^4)$, respectively.
They conjectured that the latter two bounds hold for all subcubic
graphs.  The one-matching conjecture was subsequently confirmed, in
the stronger setting of subcubic multigraphs, by Liu, Santana, and
Short~\cite{LSS1}. For two
matching colors, Liu and Yu~\cite{LY1} proved that every connected subcubic graph
with more than $70$ vertices is $(1^2,2^4)$-packing edge-colorable. 

For planar graphs, Kostochka, Li, Ruksasakchai, Santana, Wang, and Yu~\cite{KLRSWY1}
proved that every subcubic planar multigraph is
$(2^9)$-packing edge-colorable, and it is sharp. Hocquard,
Lajou, and Lu\v{z}ar~\cite{HLL1} proposed the following more refined conjecture:
every subcubic planar graph admits a $(1,2^6)$-packing edge-coloring
and a $(1^2,2^3)$-packing edge-coloring. Their proposed one-matching bound, if
true, is best possible, since they constructed a subcubic planar graph of girth $3$ that has no $(1,2^5)$-packing edge-coloring. The
conjecture has been confirmed for outerplanar graphs by Li, Li, and
Liu~\cite{LLL1}, who proved the stronger statement that every subcubic
outerplanar graph is $(1,2^5)$-packing edge-colorable and
$(1^2,2^3)$-packing edge-colorable. Both results are also sharp within
that class. Chen, Deng, Nan, Tan, and Zhou~\cite{CDNTZ1} recently showed that every claw-free subcubic graph without triangular prism components is $(1,2^5)$-packing edge-colorable.

Large girth provides a natural measure of how closely a planar graph
resembles a forest.  At the acyclic extreme, every subcubic tree is
$(1,2^4)$-packing edge-colorable~\cite{HLL1}. Liu, Yang, and Zhang~\cite{LYZ1} proved every subcubic planar graph with girth at least $20$ has a $(1^2,2^2)$-packing edge-coloring. Previous results on strong
edge-coloring already imply general, though nonoptimal, girth
thresholds for the $S$-packing edge-coloring problem.  Hocquard, Montassier, Raspaud, and
Valicov~\cite{HMRV1} proved that every subcubic planar graph of girth at least $14$
has a strong $6$-edge-coloring.  Similarly, DeOrsey, Ferrara, Graber, Hartke, Nelsen, Sullivan, Jahanbekam, Lidick\'y, Stolee, and White~\cite{DFGHNSJLSW1} proved that girth at least $30$ guarantees a strong
$5$-edge-coloring.
Since an induced matching is also a matching, these results imply,
respectively, a $(1,2^5)$-packing edge-coloring for girth at least
$14$ and a $(1,2^4)$-packing edge-coloring for girth at least $30$.
However, strong edge-colorings do not exploit the additional
flexibility supplied by the matching color, and substantially smaller
thresholds may therefore be expected.  

Motivated by these observations, let $k_1$ be the least positive
integer $g$ such that every subcubic planar graph of girth at least
$g$ is $(1,2^5)$-packing edge-colorable.  Similarly, let $k_2$
be the least positive integer $g$ such that every subcubic planar
graph of girth at least $g$ is $(1,2^4)$-packing edge-colorable.
We first notice that no analogous threshold exists with only three
induced-matching colors. It was shown in~\cite{LYZ1} that there are subcubic planar graphs with arbitrarily large girth that are not $(1,2^3)$-packing edge-colorable.

Our main results show that the two thresholds $k_1$ and $k_2$ satisfy $5\le k_1\le 12$ and $ 6\le k_2\le 16$.

\begin{theorem}\label{thm1}
Every subcubic planar graph with girth at least $12$ has a $(1,2^5)$-packing edge-coloring.   
\end{theorem}

\begin{theorem}\label{thm2}
Every subcubic planar graph with girth at least $16$ has a $(1,2^4)$-packing edge-coloring.   
\end{theorem}

Explicit planar constructions in Section~\ref{examples} give
the corresponding lower bounds. In fact, we showed results in forms of maximum average degree (mad). We proved every subcubic graph $G$ with $mad(G) < \frac{12}{5}$ is $(1,2^5)$-packing edge-colorable. Furthermore, we showed every subcubic graph $G$ with $mad(G) < \frac{9}{4}$ is $(1,2^4)$-packing edge-colorable.

\section{Lower Bounds for $k_1$ and $k_2$}\label{examples}

\begin{example}\label{example1}
Let $G_1$ be the graph shown in Figure \ref{counterexample_1}. Note that $G_1$ is subcubic, planar, and has girth $4$. We show that $G_1$ has no $(1, 2^5)$-packing edge-coloring.
\end{example}

\begin{proof}
Note that $G_1$ has $11$ vertices and $16$ edges. Thus, a maximum matching can have size at most $\lfloor \frac{11}{2} \rfloor = 5$. 

One can easily see that every induced matching in $G_1$ has size at most $2$.  Therefore, each $2$-color can be used at most twice. Since a $1$-color and five $2$-colors can cover at most $5 + 5 \times 2 = 15$ edges, and $G_1$ has $16$ edges, it has no $(1, 2^5)$-packing edge-coloring.
\end{proof}

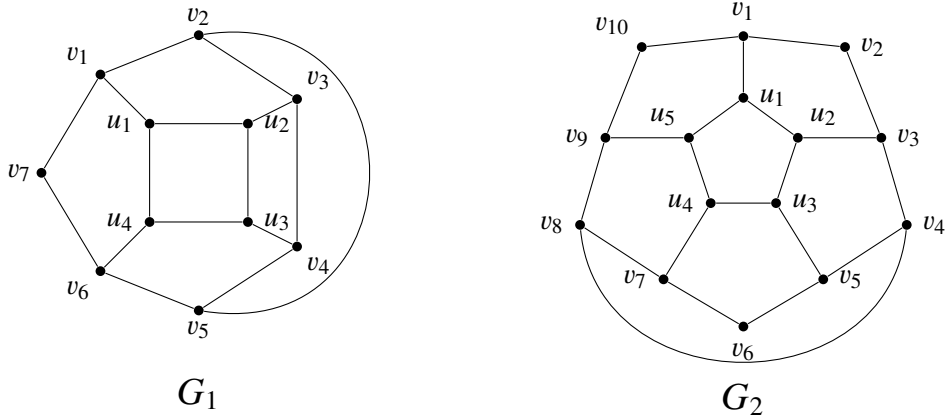
\begin{figure}[htbp]
    \centering
    \begin{tikzpicture}[
        scale=0.65,
        baseline=(current bounding box.center),
        vertex/.style={circle, fill=black, inner sep=1.3pt}
    ]
        \node[vertex] (u1) at (-1, 1) {};  \node[left=0.08cm]  at (u1) {$u_1$};
        \node[vertex] (u2) at (1, 1) {};   \node[right=0.08cm] at (u2) {$u_2$};
        \node[vertex] (u3) at (1, -1) {};  \node[right=0.08cm] at (u3) {$u_3$};
        \node[vertex] (u4) at (-1, -1) {}; \node[left=0.08cm]  at (u4) {$u_4$};

        \node[vertex] (v1) at (-2, 2) {};     \node[above left] at (v1) {$v_1$};
        \node[vertex] (v2) at (0, 2.8) {};    \node[above]      at (v2) {$v_2$};
        \node[vertex] (v3) at (2, 1.5) {};    \node[above right] at (v3) {$v_3$};
        \node[vertex] (v4) at (2, -1.5) {};   \node[below right] at (v4) {$v_4$};
        \node[vertex] (v5) at (0, -2.8) {};   \node[below]      at (v5) {$v_5$};
        \node[vertex] (v6) at (-2, -2) {};    \node[below left] at (v6) {$v_6$};
        \node[vertex] (v7) at (-3.2, 0) {};   \node[left]       at (v7) {$v_7$};

        \draw (u1) -- (u2) -- (u3) -- (u4) -- (u1); 
        \draw (v1) -- (v2) -- (v3) -- (v4) -- (v5) -- (v6) -- (v7) -- (v1); 
        \draw (u1) -- (v1) (u2) -- (v3) (u3) -- (v4) (u4) -- (v6); 

        \draw (v2) .. controls (4.6, 3.5) and (4.6, -3.5) .. (v5);

        \node[font=\Large] at (0, -4.5) {$G_1$};
    \end{tikzpicture}
    \hspace{1.2cm}
    \begin{tikzpicture}[
        scale=0.48,
        baseline=(current bounding box.center),
        vertex/.style={circle, fill=black, inner sep=1.3pt}
    ]
        \node[vertex] (u1) at (0, 1.5) {};    \node[right=0.08cm]        at (u1) {$u_1$};
        \node[vertex] (u2) at (1.5, 0.4) {};  \node[above right=0.04cm] at (u2) {$u_2$};
        \node[vertex] (u3) at (0.9, -1.4) {}; \node[right=0.08cm]        at (u3) {$u_3$};
        \node[vertex] (u4) at (-0.9, -1.4) {};\node[left=0.08cm]         at (u4) {$u_4$};
        \node[vertex] (u5) at (-1.5, 0.4) {}; \node[above left=0.04cm]  at (u5) {$u_5$};

        \node[vertex] (v1) at (0, 3.2) {};    \node[above=0.08cm]        at (v1) {$v_1$};
        \node[vertex] (v2) at (2.8, 2.9) {};  \node[right=0.08cm]        at (v2) {$v_2$};
        \node[vertex] (v3) at (3.8, 0.4) {};  \node[right=0.08cm]        at (v3) {$v_3$};
        \node[vertex] (v4) at (4.5, -2.0) {}; \node[right=0.08cm]        at (v4) {$v_4$};
        \node[vertex] (v5) at (2.2, -3.5) {}; \node[right=0.08cm]        at (v5) {$v_5$};
        \node[vertex] (v6) at (0, -4.8) {};   \node[below=0.08cm]        at (v6) {$v_6$};
        \node[vertex] (v7) at (-2.2, -3.5) {};\node[left=0.08cm]         at (v7) {$v_7$};
        \node[vertex] (v8) at (-4.5, -2.0) {};\node[left=0.08cm]         at (v8) {$v_8$};
        \node[vertex] (v9) at (-3.8, 0.4) {}; \node[left=0.08cm]         at (v9) {$v_9$};
        \node[vertex] (v10) at (-2.8, 2.9) {};\node[above left=0.04cm]  at (v10) {$v_{10}$};

        \draw (u1) -- (u2) -- (u3) -- (u4) -- (u5) -- (u1); 
        \draw (v10) -- (v1) -- (v2) -- (v3) -- (v4) -- (v5) -- (v6) -- (v7) -- (v8) -- (v9) -- (v10); 
        \draw (u1) -- (v1);
        \draw (u2) -- (v3);
        \draw (u3) -- (v5);
        \draw (u4) -- (v7);
        \draw (u5) -- (v9);

        \draw (v8) .. controls (-4.0, -7.0) and (4.0, -7.0) .. (v4);
        
        \node[font=\Large] at (0, -6.8) {$G_2$};
    \end{tikzpicture}
    \vspace{-2mm}
    \caption{Subcubic planar graphs $G_1,G_2$ with girth $4,5$ and has no $(1, 2^5)$- and no $(1, 2^4)$-packing edge-coloring respectively.}
    \label{counterexample_1}
\end{figure}

\begin{example}
Let $G_2$ be the graph shown in Figure \ref{counterexample_1}. Note that $G_2$ is subcubic, planar, and has girth $5$. We show that $G_2$ has no $(1, 2^4)$-packing edge-coloring.
\end{example}

\begin{proof}
By the definition of a $(1, 2^4)$-packing edge coloring, the $1$-color can be used either exactly once or exactly twice on the 5-cycle $C = u_1u_2u_3u_4u_5u_1$.

\textbf{Case 1:} The $1$-color appears exactly once on the cycle $C$. By symmetry, the edge colored with $1$ can be $u_3u_4$ or $u_4u_5$ or $u_1u_5$. We obtain a contradiction in each case. 

\textbf{Case 1.1:} $u_3u_4$ is colored with $1$. We may assume $u_4u_5, u_1u_5, u_1u_2$, and $u_2u_3$ are colored with $2_1, 2_2, 2_3$, and $2_4$, respectively. Then $u_3v_5$, and $u_4v_7$ must be colored with $2_2, 2_3$. The edge $u_2v_3$ can only use $1$ or $2_1$. 
In the former subcase, $v_2v_3, v_3v_4$ must be colored with $2_2, 2_1$. Then $v_5v_6$ must be colored with $1$, and $v_4v_5$ must be colored with $2_3$. This forces $v_4v_8, v_6v_7$ to use $1, 2_4$. However, $v_8v_9$ has no color to use. This is a contradiction.
In the latter subcase, $v_3v_4$ can only use $1$. Then $v_4v_5, v_4v_8$ must be colored with $2_3, 2_4$. Then $v_6v_7$ can only use $1$. This forces $v_5v_6, v_7v_8$ to be colored with $2_1, 2_2$, respectively. Consequently, $v_8v_9$ can only use $1$. However, $u_5v_9$ has no color to use. This is again a contradiction.

\textbf{Case 1.2:} $u_4u_5$ is colored with $1$. We may assume $u_1u_5, u_1u_2, u_2u_3$, and $u_3u_4$ are colored with $2_1, 2_2, 2_3$, and $2_4$. Then $u_2v_3, u_4v_7$, and $u_5v_9$ must be colored with $1, 2_2$, and $2_3$. The edge $u_3v_5$ can only use $1$ or $2_1$.
In the former subcase, $v_4v_5$ and $v_5v_6$ must be colored with $2_2$ and $2_1$. Then $v_3v_4$ can only use $2_4$. This forces $v_4v_8$ to be colored with $1$. However, $v_7v_8$ has no color to use. This is a contradiction.
In the latter subcase, $v_5v_6$ and $v_4v_5$ must be colored with $1$ and $2_2$. Then $v_3v_4$ can only use $2_4$. This forces $v_4v_8$ to use $1$. However, $v_8v_9$ has no color to use. This is again a contradiction.

\textbf{Case 1.3:} $u_1u_5$ is colored with $1$. We may assume $u_1u_2, u_2u_3, u_3u_4$, and $u_4u_5$ are colored with $2_1, 2_2, 2_3$, and $2_4$. Then $u_3v_5, u_5v_9$ must be colored with $1, 2_2$. The edge $u_4v_7$ can only use $1$ or $2_1$.
In the former subcase, $v_6v_7$ and $v_7v_8$ must be colored with $2_2$ and $2_1$. This forces $v_4v_5$ and $v_5v_6$ to be colored with $2_1$ and $2_4$. However, $v_4v_5$ and $v_7v_8$ are both colored with $2_1$. This is a contradiction.
In the latter subcase, $v_7v_8$ and $v_6v_7$ must be colored with $1$ and $2_2$. This forces $v_5v_6$ to be colored with $2_4$. However, neither $v_4v_8$ nor $v_8v_9$ can use $2_4$. This is again a contradiction.

\textbf{Case 2:} The $1$-color appears exactly twice on $C$. By symmetry, the edges colored with $1$ can be $u_3u_4$ and $u_1u_5$, $u_2u_3$ and $u_4u_5$, or $u_1u_2$ and $u_4u_5$. Since our proof does not use the edge $v_4v_8$, all three cases are symmetric. We may assume $u_3u_4$ and $u_1u_5$ are colored with $1$, and $u_1u_2, u_2u_3$, and $u_4u_5$ are colored with $2_1, 2_2$, and $2_3$. Then $u_1v_1$ and $u_3v_5$ must both be colored with $2_4$. The edge $u_2v_3$ can only use $1$ or $2_3$. In the former subcase, neither $v_2v_3$ nor $v_3v_4$ can use $2_4$. This is a contradiction. In the latter subcase, both $v_2v_3$ and $v_3v_4$ have $1$ as the only available color. This is again a contradiction. 
\end{proof}

\section{Proof of Theorem~\ref{thm1}}
A {\em $k$-chain} is a path on $k$ vertices such that each vertex has degree exactly $2$ in $G$. A {\em $k$-thread} is a path with end vertices of degree at least three, and the internal vertices forms a $k$-chain. A vertex of degree $d$ is denoted by {\em $d$-vertex}. We say a vertex $u$ {\em sees} a $1$-color $c_1$ if an edge incident on $u$ is using $c_1$. We say a vertex $u$ {\em sees} a $2$-color $c_2$ if an edge $e = u_1u_2$ is assigned with $c_2$ and either $u_1$ or $u_2$ are at distance at most one from $u$. We say an edge $e$ {\em sees} a $1$-color $c_1$ ($2$-color $c_2$) if there is an edge colored with $c_1$ ($c_2$) and it is within distance one (two) from $e$.

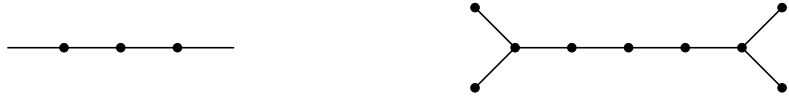
\begin{figure}[ht]
 \begin{center}
    \hfill
   \begin{tikzpicture}[
       scale=0.5,
       baseline=0,
       dot/.style={circle, draw=black, fill=black, inner sep=0pt, minimum size=3.2pt},
       elabel/.style={inner sep=1.5pt}
   ]
       \draw[semithick] (-1.5,0) -- (4.5,0);

       \node[dot] at (0,0) {};
       \node[dot] at (1.5,0) {};
       \node[dot] at (3.0,0) {};
   \end{tikzpicture} \hfill
   \begin{tikzpicture}[
       scale=0.5,
       baseline=0,
       dot/.style={circle, draw=black, fill=black, inner sep=0pt, minimum size=3.2pt},
       elabel/.style={inner sep=1.5pt}
   ]
       \coordinate (h0) at (0,0);
       \coordinate (h1) at (1.5,0);
       \coordinate (h2) at (3.0,0);
       \coordinate (h3) at (4.5,0);
       \coordinate (h4) at (6.0,0);

       \path (h0) ++(135:1.5) coordinate (left_up);
       \path (h0) ++(225:1.5) coordinate (left_down);

       \path (h4) ++(45:1.5) coordinate (right_up);
       \path (h4) ++(-45:1.5) coordinate (right_down);

       \draw[semithick] (left_up) -- (h0) -- (left_down);
       \draw[semithick] (h0) -- (h4);
       \draw[semithick] (right_up) -- (h4) -- (right_down);

       \node[dot] at (left_up) {};
       \node[dot] at (left_down) {};

       \node[dot] at (h0) {};
       \node[dot] at (h1) {};
       \node[dot] at (h2) {};
       \node[dot] at (h3) {};
       \node[dot] at (h4) {};

       \node[dot] at (right_up) {};
       \node[dot] at (right_down) {};
   \end{tikzpicture} 
   \hfill\,
\caption{Examples: a $3$-chain and a $3$-thread}
\end{center}
\vspace{-8mm}
\label{3-chain and 3-thread}
\end{figure}

We use $1,2_1, 2_2, 2_3, 2_4, 2_5$ to denote the $1$-color and the five $2$-colors. 
In order to extend the coloring when proving that $G$ contains no $1$-vertices, we introduce a good coloring condition as follows.
In this section, we say that a $(1,2^5)$-packing edge-coloring is {\em good} if it further satisfies the following condition:

\vspace{2mm}

\textbf{Condition I:} for each $2$-vertex $u$, if $u$ sees $1$-color, then $u$ does not see all the five $2$-colors $2_1, 2_2, 2_3, 2_4, 2_5$. 

\begin{figure}[ht]
\centering
\begin{tikzpicture}[scale=0.7,
    vertex/.style={circle, fill=black, inner sep=1.5pt},
    every node/.style={font=\small}
]

\node[vertex] (x) at (-2,0) {};
\node[vertex] (u) at (0,0) {};
\node[vertex] (y) at (2,0) {};

\node[vertex] (a) at (-3,1.4) {};
\node[vertex] (b) at (-3,-1.4) {};
\node[vertex] (c) at (3,1.4) {};
\node[vertex] (d) at (3,-1.4) {};

\node[below=3pt] at (x) {$x$};
\node[below=3pt] at (u) {$u$};
\node[below=3pt] at (y) {$y$};

\node[left=3pt] at (a) {$x_1$};
\node[left=3pt] at (b) {$x_2$};
\node[right=3pt] at (c) {$y_1$};
\node[right=3pt] at (d) {$y_2$};

\draw (x) -- node[above] {\textcolor{blue}{$2_1$}} (u);
\draw (u) -- node[above] {\textcolor{blue}{$1$}} (y);

\draw (x) -- node[left] {\textcolor{blue}{$2_2$}} (a);
\draw (x) -- node[left] {\textcolor{blue}{$2_3$}} (b);

\draw (y) -- node[right] {\textcolor{blue}{$2_4$}} (c);
\draw (y) -- node[right] {\textcolor{blue}{$2_5$}} (d);
\end{tikzpicture}
\caption{The $2$-vertex $u$ sees the $1$-color and all five $2$-colors
$2_1,2_2,2_3,2_4,2_5$, violating Condition~I.
Thus, the $(1,2^5)$-packing edge-coloring is not good.}
\label{fig:not-good-1-thread}
\end{figure}
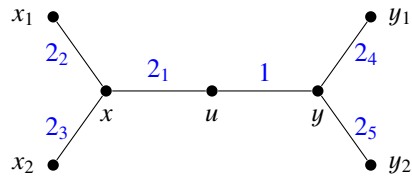
\vspace{2mm}

We show the following stronger theorem, which implies Theorem~\ref{thm1}.

\begin{thm}\label{Thm 1}
Every subcubic graph $G$ with  $\text{mad}(G) < \frac{12}{5}$ has a good coloring.    
\end{thm}

By a standard application of Euler's formula, every planar graph with girth at least $g$ satisfies $\text{mad}(G) < \frac{2g}{g-2}$. We obtain the following result as a corollary of Theorem~\ref{Thm 1}.

\begin{cor}
Every subcubic planar graph with girth at least $12$ has a good coloring.    
\end{cor}

\textbf{Proof of Theorem~\ref{Thm 1}:} Suppose not. Let $G$ be a subcubic graph with  $\text{mad}(G) < \frac{12}{5}$ and has no good coloring. We further require that $G$ has the minimum $|V(G)| + |E(G)|$. 

Ruling out $1$-vertices is the most crucial structural foundation for our discharging argument. However, since the extended coloring must strictly preserve Condition I at all $2$-vertices, a detailed case analysis is required.

\begin{lemma}\label{mindegree Thm1}
$\delta(G) \ge 2$.    
\end{lemma}

\begin{proof}
Suppose not, i.e., there is a $1$-vertex $u$ with the unique neighbour $u_1$. We delete $uu_1$ to obtain a subcubic graph $G'$ with $\text{mad}(G') < \frac{12}{5}$. By the minimality of $G$, $G'$ has a good coloring $f$. We extend $f$ to $G$ to obtain a contradiction. We may assume that $u_1$ is a $3$-vertex, since otherwise we can color $uu_1$ with an available $2$-color. This is a contradiction. Let $N(u_1) = \{u, u_2, u_3\}$. Note that the case where both $u_2$ and $u_3$ are $1$-vertices is trivial, since $uu_1$ would see at most two colors in $G'$.

\textbf{Case 1:} At least one of $u_2, u_3$ is a $2$-vertex. By symmetry, we may assume $u_2$ is a $2$-vertex. Let $N(u_2) = \{u_1, u_4\}$. First, suppose $u_3$ is a $1$-vertex. By symmetry, $f(u_1u_2), f(u_1u_3) = 2_1, 2_2$ or $2_1, 1$ or $1, 2_1$. In the first case, we color $uu_1$ with $1$ to obtain a good coloring. In the middle case, we color $uu_1$ with an available $2$-color to obtain a good coloring. In the last case, we may assume $f(u_2u_4) = 2_2$. We can color $uu_1$ with an available $2$-color. The only bad case is when $u_4$ is a $3$-vertex and $u_2$ sees all $6$ colors. In this case, we recolor $uu_1$ with a different available $2$-color (it originally has at least two available $2$-colors).

Next, suppose $u_3$ is a $2$-vertex. Let $N(u_3) = \{u_1, u_5\}$. By symmetry, $f(u_1u_2), f(u_1u_3) = 2_1, 2_2$ or $1, 2_1$. In the former case, we color $uu_1$ with $1$ to obtain a good coloring. 
In the latter case, We may assume $f(u_2u_4) = 2_2$. We can color $uu_1$ with an available $2$-color. The only bad case is when $u_4$ is a $3$-vertex and $u_2$ sees all $6$ colors. In this case, we recolor $uu_1$ with a different available $2$-color (it originally has at least two available $2$-colors).

Finally, suppose $u_3$ is a $3$-vertex, let $N(u_3) = \{u_1, u_5, u_6\}$. By symmetry, $f(u_1u_2), f(u_1u_3) = 2_1, 2_2$ or $2_1, 1$ or $1, 2_1$. 
In the first case, we color $uu_1$ with $1$ to obtain a good coloring. In the middle case, we color $uu_1$ with an available $2$-color to obtain a good coloring. 
In the last case, We may assume $f(u_2u_4) = 2_2$. We can color $uu_1$ with an available $2$-color. The only bad case is when $u_4$ is a $3$-vertex and $u_2$ sees all $6$ colors. We claim $2_2 \in \{f(u_3u_5), f(u_3u_6)\}$ since otherwise we switch the colors of $u_1u_2$ and $u_2u_4$, and color $uu_1$ with $1$ to obtain a good coloring. However, this means $uu_1$ has two available $2$-colors to use. If an available $2$-color violates Condition I for $u_2$, then use the other $2$-color. This would always guarantee a good coloring.

\textbf{Case 2:} At least one of $u_2, u_3$ is a $3$-vertex. By symmetry, we may assume $u_2$ is a $3$-vertex.  If $u_3$ is a $1$-vertex, we observe that $uu_1$ sees at most four colors, and we can always extend $f$ to $G$ by coloring $uu_1$ with an available color. If $u_3$ is a $2$-vertex, by symmetry, this case has already been discussed in Case 1. If $u_3$ is a $3$-vertex, by the definition of a good coloring, $u_1$ sees at most five colors, and we can always extend $f$ to $G$ (and remain as a good coloring) by coloring $uu_1$ with an available color.
\end{proof}

\begin{figure}[ht]
 \begin{center}
    \hfill
   \begin{tikzpicture}[
       scale=0.45,
       baseline=0,
       dot/.style={circle, draw=black, fill=black, inner sep=0pt, minimum size=3pt},
        execute at begin picture={\def\scriptsize{\fontsize{7}{8.5}\selectfont}},
       elabel/.style={inner sep=1.5pt}
   ]
       \draw[semithick] (-1.5,0) -- (4.5,0);

       \node[dot] at (0,0) {};
       \node[dot] at (1.5,0) {};
       \node[dot] at (3.0,0) {};

       \node[font=\scriptsize, below=1pt] at (0,0) {$v_1$};
       \node[font=\scriptsize, below=1pt] at (1.5,0) {$v_2$};
       \node[font=\scriptsize, below=1pt] at (3.0,0) {$v_3$};
       \node[font=\scriptsize, text=blue, above=0.5pt] at (-0.75,0) {$1$};
       \node[font=\scriptsize, text=red, above=0.5pt] at (0.75,0) {$2_1$};
       \node[font=\scriptsize, text=red, above=0.5pt] at (2.25,0) {$2_2$};
       \node[font=\scriptsize, text=blue, above=0.5pt] at (3.75,0) {$1$};
       \node[font=\footnotesize] at (1.5,-1.5) {(1)};
   \end{tikzpicture} \hfill
   \begin{tikzpicture}[
       scale=0.45,
       baseline=0,
       dot/.style={circle, draw=black, fill=black, inner sep=0pt, minimum size=3pt},
        execute at begin picture={\def\scriptsize{\fontsize{7}{8.5}\selectfont}},
       elabel/.style={inner sep=1.5pt}
   ]
       \coordinate (junction) at (0,0);
       \path (junction) ++(135:1.5) coordinate (d_up);
       \path (junction) ++(225:1.5) coordinate (d_down);
       \coordinate (t1) at (1.5, 0);
       \coordinate (t2) at (3.0, 0);
       \coordinate (t3) at (4.5, 0);
       \coordinate (t4) at (6.0, 0);
       \coordinate (t_end) at (7.5, 0); 
       \coordinate (v_dot) at (4.5, 1.8);
       \coordinate (v_end) at (4.5, 3.0);

       \draw[semithick] (d_up) -- (junction) -- (d_down);
       \draw[semithick] (junction) -- (t_end);
       \draw[semithick] (t3) -- (v_end);

       \node[dot] at (d_up) {};
       \node[dot] at (d_down) {};
       \node[dot] at (junction) {};
       \node[dot] at (t1) {};
       \node[dot] at (t2) {};
       \node[dot] at (t3) {};
       \node[dot] at (t4) {};
       \node[dot] at (v_dot) {};

       \node[font=\scriptsize, above right=0pt] at (d_up) {$v_4$};
       \node[font=\scriptsize, below right=0pt] at (d_down) {$v_5$};
       \node[font=\scriptsize, below=1pt] at (junction) {$v_3$};
       \node[font=\scriptsize, below=1pt] at (t1) {$v_2$};
       \node[font=\scriptsize, below=1pt] at (t2) {$v_1$};
       \node[font=\scriptsize, below=2pt] at (t3) {$u$};
       \node[font=\scriptsize, below=1pt] at (t4) {$w_1$};
       \node[font=\scriptsize, right=1pt] at (v_dot) {$u_1$};
     
\node[font=\scriptsize, text=blue, below right=3pt and -4pt] at (d_up) {$2_3$};
\node[font=\scriptsize, text=blue, above right=3pt and -4pt] at (d_down) {$2_4$};
       \path (junction) -- node[font=\scriptsize, text=blue, above=.8pt] {$1$} (t1);
       \path (t1) -- node[font=\scriptsize, text=red, above=.1pt] {$2_5$} (t2);
       \path (t2) -- node[font=\scriptsize, text=blue, above=.3pt, pos=.38] {$1$} (t3);
       \path (t3) -- node[font=\scriptsize, text=blue, above=.3pt, pos=.5] {$2_2$} (t4);
       \path (t3) -- node[font=\scriptsize, text=blue, left=0.3pt, pos=.60] {$2_1$} (v_dot);
       \node[font=\footnotesize] at (3.0,-2.2) {(2)};
   \end{tikzpicture} \hfill
   \begin{tikzpicture}[
       scale=0.45,
       baseline=0,
       dot/.style={circle, draw=black, fill=black, inner sep=0pt, minimum size=3pt},
        execute at begin picture={\def\scriptsize{\fontsize{7}{8.5}\selectfont}},
       elabel/.style={inner sep=1.5pt}
   ]
       \coordinate (h0) at (0,0);
       \coordinate (h1) at (1.5,0);
       \coordinate (h2) at (3.0,0);
       \coordinate (h3) at (4.5,0);
       \coordinate (h4) at (6.0,0);
       \coordinate (h5) at (7.5,0);
       \coordinate (h6) at (9.0,0);

       \path (h0) ++(135:1.5) coordinate (left_up);
       \path (h0) ++(225:1.5) coordinate (left_down);

       \path (h6) ++(45:1.5) coordinate (right_up);
       \path (h6) ++(-45:1.5) coordinate (right_down);

       \path (h3) ++(90:1.8) coordinate (v_dot);
       \path (v_dot) ++(135:1.5) coordinate (v_left);
       \path (v_dot) ++(45:1.5) coordinate (v_right);

       \draw[semithick] (left_up) -- (h0) -- (left_down);
       \draw[semithick] (h0) -- (h6);
       \draw[semithick] (right_up) -- (h6) -- (right_down);
       \draw[semithick] (h3) -- (v_dot);
       \draw[semithick] (v_left) -- (v_dot) -- (v_right);

       \node[dot] at (left_up) {};
       \node[dot] at (left_down) {};
       \node[dot] at (right_up) {};
       \node[dot] at (right_down) {};
       \node[dot] at (v_left) {};
       \node[dot] at (v_right) {};

       \node[dot] at (h0) {};
       \node[dot] at (h1) {};
       \node[dot] at (h2) {};
       \node[dot] at (h3) {};
       \node[dot] at (h4) {};
       \node[dot] at (h5) {};
       \node[dot] at (h6) {};
       \node[dot] at (v_dot) {};

       \node[font=\scriptsize, above right=0pt] at (left_up) {$v_4$};
       \node[font=\scriptsize, below right=0pt] at (left_down) {$v_5$};
       \node[font=\scriptsize, below=1pt] at (h0) {$v_3$};
       \node[font=\scriptsize, below=1pt] at (h1) {$v_2$};
       \node[font=\scriptsize, below=1pt] at (h2) {$v_1$};
       \node[font=\scriptsize, below=2pt] at (h3) {$u$};
       \node[font=\scriptsize, below=1pt] at (h4) {$w_1$};
       \node[font=\scriptsize, below=1pt] at (h5) {$w_2$};
       \node[font=\scriptsize, below=1pt] at (h6) {$w_3$};
       \node[font=\scriptsize, right=1pt] at (v_dot) {$u_1$};
       \node[font=\scriptsize, above left=0pt] at (v_left) {$u_3$};
       \node[font=\scriptsize, above right=0pt] at (v_right) {$u_2$};
       \node[font=\scriptsize, above left=0pt] at (right_up) {$w_4$};
       \node[font=\scriptsize, below left=0pt] at (right_down) {$w_5$};
        \node[font=\scriptsize, text=blue, below right=3pt and -4pt] at (left_up) {$2_4$};
\node[font=\scriptsize, text=blue, above right=3pt and -4pt] at (left_down) {$1$};
       \path (h0) -- node[font=\scriptsize, text=blue, inner sep=0pt, above=.1pt] {$2_1$} (h1);
       \path (h1) -- node[font=\scriptsize, text=red, inner sep=0pt, above=.1pt] {$2_5$} (h2);
       \path (h2) -- coordinate[pos=.5] (figFourChangeA) (h3);
       \node[font=\scriptsize, text=blue, inner sep=0pt, anchor=south east, xshift=-.5pt, yshift=.8pt] (figFourOldA) at (figFourChangeA) {$2_1$};
       \node[font=\scriptsize, inner sep=0pt, text=red, anchor=south west, xshift=.8pt, yshift=1.5pt] at (figFourChangeA) {$1$};
       \draw[semithick] (figFourOldA.south west) -- (figFourOldA.north east);
       \path (h3) -- coordinate[pos=.5] (figFourChangeB) (h4);
       \node[font=\scriptsize, text=blue, inner sep=0pt, anchor=south east, xshift=-.8pt, yshift=1.5pt] (figFourOldB) at (figFourChangeB) {$1$};
       \node[font=\scriptsize, inner sep=0pt, text=red, anchor=south west, xshift=.8pt, yshift=.8pt] at (figFourChangeB) {$2_3$};
       \draw[semithick] (figFourOldB.south west) -- (figFourOldB.north east);
       \path (h4) -- coordinate[midway] (figFourChangeC) (h5);
       \node[font=\scriptsize, text=blue, inner sep=0pt, anchor=south east, xshift=-.8pt, yshift=.8pt] (figFourOldC) at (figFourChangeC) {$2_3$};
       \node[font=\scriptsize, inner sep=0pt, text=red, anchor=south west, xshift=.8pt, yshift=1.5pt] at (figFourChangeC) {$1$};
       \draw[semithick] (figFourOldC.south west) -- (figFourOldC.north east);
       \path (h5) -- node[font=\scriptsize, text=blue, inner sep=0pt, above=.1pt] {$2_1$} (h6);
       \path (h3) -- node[font=\scriptsize, text=blue, left=0.1pt, pos=.70] {$2_2$} (v_dot);
        \path (v_left) -- node[font=\scriptsize, text=blue, fill=white, inner sep=.3pt, below left=1.2pt, pos=.38] {$2_5$} (v_dot);
        \path (v_right) -- node[font=\scriptsize, text=blue, fill=white, inner sep=.3pt, below right=1.2pt, pos=.38] {$2_4$} (v_dot);
       \node[font=\scriptsize, text=blue, below left=3pt and -4pt] at (right_up) {$1$};
\node[font=\scriptsize, text=blue, above left=3pt and -4pt] at (right_down) {$2_2$};
       \node[font=\footnotesize] at (4.5,-2.2) {(3)};
   \end{tikzpicture} 
   \hfill\,
\caption{Lemma~\ref{no $3$-chain}, Lemma~\ref{no211} and Lemma~\ref{no220}}
\label{lemma 3.4, 3.5 and 3.6}
\end{center}
\vspace{-8mm}
\end{figure}
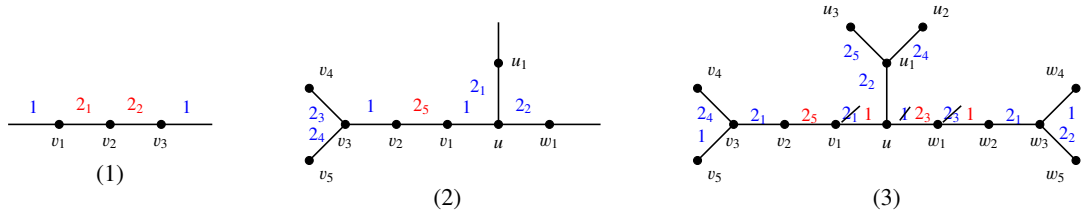

\begin{lemma}\label{no $3$-chain}
There is no $3$-chain.    
\end{lemma}

\begin{proof}
Suppose not, say $v_1v_2v_3$ is a $3$-chain. We delete $v_2$ to obtain a subcubic graph $G'$ with $\text{mad}(G') < \frac{12}{5}$. Let $N(v_1)=\{u_1, v_2\}$ and $N(v_3)=\{v_2, u_2\}$. By the minimality of $G$, $G'$ has a good coloring $f$. By symmetry, $f(u_1v_1), f(u_2v_3) = 1, 1$ or $1, 2_1$ or $2_1, 2_1$ or $2_1, 2_2$. We extend $f$ by first coloring $v_1v_2$ and then coloring $v_2v_3$. In the last three cases, since $v_1v_2$ sees at most four colored edges, we can always color $v_1v_2$ with an available $2$-color, and then color $v_2v_3$ with $1$ to obtain a good coloring. In the first case, since $v_1v_2$ sees at most two $2$-colors, we can always color $v_1v_2$ with an available $2$-color; subsequently, since $v_2v_3$ sees at most three $2$-colors, we can always color $v_2v_3$ with an available $2$-color to obtain a good coloring (see Figure~\ref{lemma 3.4, 3.5 and 3.6} (1)).
\end{proof}

By Lemma \ref{mindegree Thm1} and Lemma \ref{no $3$-chain}, any $2$-chain in $G$ must be a $2$-thread.

\begin{lemma}\label{no211}
If a $3$-vertex $u$ is adjacent to a $2$-thread, then it cannot be adjacent to another two $1$-chains.
\end{lemma}

\begin{proof}
Suppose not, say $u$ is adjacent to a $2$-thread $uv_1v_2v_3$, where $v_3$ is a $3$-vertex, and two 1-chains $w_1$ and $u_1$. We delete $v_1v_2$ to obtain a subcubic graph $G'$ with $\text{mad}(G') < \frac{12}{5}$. By the minimality of $G$, $G'$ has a good coloring $f$. Let $N(v_3) = \{v_2, v_4, v_5\}$, $N(w_1) = \{u, w_2\}$ and $N(u_1) = \{u, u_2\}$.
By symmetry, $f(uv_1), f(v_2v_3) = 1, 1$ or $1, 2_1$ or $2_1, 1$ or $2_1, 2_2$ or $2_1, 2_1$.  

In the first case, since $v_1v_2$ sees at most four $2$-colors, we can always color $v_1v_2$ with an available $2$-color to obtain a good coloring(see Figure~\ref{lemma 3.4, 3.5 and 3.6} (2)). In the second case, we may assume that one of $f(v_3v_4)$ and $f(v_3v_5)$ is $1$, otherwise we can recolor $v_2v_3$ with $1$, which reduces to the first case. However, since $v_1v_2$ sees at most four $2$-colors, we can color $v_1v_2$ with an available $2$-color to obtain a good coloring. The third case can be proved by repeating the argument for the second case.

In the fourth case, we color $v_1v_2$ with $1$ to obtain a good coloring.

In the fifth case, we may assume that one of $f(uw_1)$ and $f(uu_1)$ is $1$, and one of $f(v_3v_4)$ and $f(v_3v_5)$ is $1$. Otherwise, we could recolor $uv_1$ or $v_2v_3$ with $1$, which reduces to one of the previous cases. By symmetry, we may assume $f(uw_1) = 1$, $f(v_3v_4) = 1$, $f(uu_1)=2_2$ and $f(w_1w_2)=2_3$.   
Note that $uv_1$ sees at most four $2$-colors, so we can recolor
$uv_1$ with an available $2$-color. It remains to check whether the
resulting coloring is a good coloring. The $2$-vertex $w_1$ is the
only vertex that may violate the conditions of a good coloring.

If $w_2$ is a $2$-vertex, then it is easy to see that $w_1$ satisfies
the conditions of a good coloring. Thus, we may assume that $w_2$ is
a $3$-vertex. Let $N(w_2)=\{w_1,w_3,w_4\}$.

The only bad case occurs when
\[
\{f(w_2w_3),f(w_2w_4)\}=\{1,2_x\}
\quad\text{and}\quad
f(u_1u_2)=2_x
\]
for some $2_x\in\{2_4,2_5\}$. In this case, we switch the colors of
$uw_1$ and $w_1w_2$, recolor $uv_1$ with $1$, and color $v_1v_2$
with an available $2$-color. This yields a good coloring, a
contradiction.
\end{proof}

\begin{lemma}\label{no220}
If a $3$-vertex $u$ is adjacent to a $2$-thread, then it cannot be adjacent to another $2$-thread and a $3$-vertex. 
\end{lemma}

\begin{proof}
Suppose not, say $u$ is adjacent to two $2$-threads $uv_1v_2v_3$ and $uw_1w_2w_3$, and another vertex $u_1$, where $v_3$ and $w_3$ are $3$-vertices. By Lemma~\ref{no211}, $u_1$ must be a $3$-vertex. We delete $v_1v_2$ to obtain a subcubic graph $G'$ with $\text{mad}(G') < \frac{12}{5}$. By the minimality of $G$, $G'$ has a good coloring $f$. Let $N(v_3) = \{v_2, v_4, v_5\}$, $N(w_3) = \{w_2, w_4, w_5\}$ and $N(u_1) = \{u, u_2, u_3\}$.
By symmetry, $f(uv_1), f(v_2v_3) = 1, 1$ or $1, 2_1$ or $2_1, 1$ or $2_1, 2_2$ or $2_1, 2_1$. 
The proofs for the first four cases are identical to that of Lemma \ref{no211}.

In the fifth case, we may assume that one of $f(uw_1)$ and $f(uu_1)$ is $1$, and one of $f(v_3v_4)$ and $f(v_3v_5)$ is $1$. Otherwise, we could recolor $uv_1$ or $v_2v_3$ with $1$, which reduces to one of the previous cases.
When $f(uw_1) = 1$, we may assume $f(uu_1) = 2_2$ and $f(w_1w_2) = 2_3$. We claim $\{f(u_1u_2), f(u_1u_3)\} = \{2_4, 2_5\}$, since otherwise $uv_1$ sees at most four $2$-colors, and we can always recolor $uv_1$ with an available $2$-color. If $f(w_2w_3) \neq 2_1$, we switch the colors of  $uv_1$ and $uw_1$, and color $v_1v_2$ with an available $2$-color. This is a contradiction. Therefore, $f(w_2w_3) = 2_1$. We switch $uw_1$ and $w_1w_2$, recolor $uv_1$ with $1$, and color $v_1v_2$ with an available $2$-color to obtain a good coloring (see Figure~\ref{lemma 3.4, 3.5 and 3.6} (3)). This is again a contradiction.

When $f(uu_1) = 1$, we may assume $f(uw_1) = 2_2$, $f(u_1u_2) = 2_3$, and $f(u_1u_3) = 2_4$. We claim $f(w_1w_2) = 2_5$, since otherwise $uv_1$ sees at most four $2$-colors, and we can always recolor $uv_1$ with an available $2$-color. Furthermore, we claim $f(w_2w_3) = 1$, since otherwise we can recolor $w_1w_2$ with $1$. However, we switch $uv_1$ and $uw_1$, and subsequently color $v_1v_2$ with $1$ to obtain a good coloring. This is a contradiction.
\end{proof}

We conclude the proof using a discharging argument. We assign an initial charge of $\text{ch}(v) = d(v) - \frac{12}{5}$ to each vertex $v \in V(G)$. Since $\text{mad}(G) < \frac{12}{5}$, we have:
\begin{equation}
\sum_{v \in V(G)} \text{ch}(v) = \sum_{v \in V(G)} \left(d(v) - \frac{12}{5}\right) = 2|E(G)| - \frac{12}{5}|V(G)| < 0.
\end{equation}

We redistribute the initial charges according to the following discharging rule:

\textbf{Rule 1:} Every $3$-vertex sends $\frac{1}{5}$ to each $2$-vertex it can reach via a path of $2$-vertices.

By Lemma \ref{mindegree Thm1}, $G$ contains no 1-vertices, so we only need to evaluate the final charges of $2$-vertices and $3$-vertices. By Lemma \ref{no $3$-chain},~\ref{no211} and~\ref{no220}, a $3$-vertex can be connected to at most three $2$-vertices via paths of $2$-vertices. Consequently, for each $3$-vertex $u$, the final charge is $\text{ch}^*(u) \ge 3 - 3 \cdot \frac{1}{5} - \frac{12}{5} = 0$. Furthermore, every $2$-vertex connects to exactly two $3$-vertices via paths of $2$-vertices. Thus, each $2$-vertex $v$ receives charges from two $3$-vertices and has a final charge of $\text{ch}^*(v) \ge 2 + 2 \cdot \frac{1}{5} - \frac{12}{5} = 0$.

Therefore, we deduce that
$$ 0 \le \sum_{v \in V(G)} \text{ch}^*(v) = \sum_{v \in V(G)} \text{ch}(v) < 0, $$
which is a contradiction. \hfill \qed

\section{Proof of Theorem~\ref{thm2}}

We use $1,2_1, 2_2, 2_3, 2_4$ to denote the $1$-color and four $2$-colors. In this section, we say that a $(1,2^4)$-packing edge-coloring $f$ is {\em good} if it further satisfies: 

\textbf{Condition 1:} The number of edges colored with 1 is maximal.

In order to extend the coloring when proving that $G$ contains no $1$-vertices, we introduce the following Condition 2.

\textbf{Condition 2:} For each $2$-vertex $u$, if $u$ sees $1$-color, then $u$ does not see all the four $2$-colors $2_1, 2_2, 2_3, 2_4$.

To avoid an unavoidable obstacle when extending the coloring in the absence of $1$-vertices, we introduce the following Condition 3.
 
\textbf{Condition 3:} Let $u_1v_1v_2u_2$ be a $2$-thread, with $u_1,u_2$ being the $3$-vertices. Let $N(u_1) = \{v_1, u_3, u_4\}$ and $N(u_2) = \{v_2, u_5, u_6\}$. If $u_1u_3, u_1u_4, u_2u_5, u_2u_6$ are colored with $2$-colors, $\{f(u_1u_3), f(u_1u_4)\} \neq \{f(u_2u_5), f(u_2u_6)\}$.

We show the following stronger theorem, which implies Theorem \ref{thm2}.

\begin{thm}\label{Thm 2}
Every subcubic planar graph with girth at least $16$ has a good coloring.    
\end{thm}

\textbf{Proof:} Suppose not. Let $G$ be a subcubic planar graph with girth at least $16$ and has no good coloring. We further require that $G$ has the minimum $|V(G)| + |E(G)|$. We may assume that $G$ is connected.

Ruling out $1$-vertices is the most crucial structural step for our discharging argument. Due to the extra restrictions imposed by Conditions~1--3, an extensive case analysis is required to successfully extend the good coloring.

\begin{lemma}\label{mindegree Thm2}
$\delta(G) \ge 2$.    
\end{lemma}

\begin{proof}
Suppose $G$ has a $1$-vertex $u$. Let $N(u) = \{u_1\}$. We know $u_1$ is a $3$-vertex, since otherwise $u_1$ is a $2$-vertex with neighbors $u, u_2$ (we are done if $u_1$ is a $1$-vertex). We delete $u$ from $G$ to obtain a subcubic graph $G'$ with girth at least $16$. By the minimality of $G$, $G'$ has a good coloring $f$. We know $f(u_1u_2)$ is either $1$ or a $2$-color ($u_2$ can be a $2$-vertex or a $3$-vertex).
In the former case, by the definition of a good coloring, we can always color $uu_1$ with an available $2$-color to obtain a good coloring. In the latter case, we can color $uu_1$ with $1$ to obtain a good coloring, which is a contradiction.
Let $N(u_1) = \{u, u_2, u_3\}$. Note that the case where both $u_2$ and $u_3$ are $1$-vertices is trivial, since $uu_1$ would see at most two colors in $G'$.

\textbf{Case 1:} At least one of $u_2, u_3$ is a $2$-vertex. By symmetry, we may assume $u_2$ is a $2$-vertex. Let $N(u_2) = \{u_1, u_4\}$. Note that if $f(u_1u_2), f(u_1u_3) = 2_1, 2_2$, then we color $uu_1$ with $1$ to obtain a good coloring. Therefore, we may assume $u_1u_2, u_1u_3$ are not colored with $2$-colors at the same time.

\textbf{Case 1.1:} $u_3$ is a $1$-vertex. By symmetry, $f(u_1u_2), f(u_1u_3) =2_1, 1$ or $1, 2_1$. In case $u_4$ is a $1$-vertex, we have a graph on $4$-vertices and can obtain a good coloring immediately. In case $u_4$ is a $3$-vertex, say $N(u_4) = \{u_2, u_5, u_6\}$, we color $uu_1$ with a $2$-color from $\{f(u_4u_5), f(u_4u_6)\}$ to obtain a good coloring. Hence, $u_4$ is a $2$-vertex. Let $N(u_4) = \{u_2,u_5\}$. In case $f(u_1u_2), f(u_1u_3) =2_1, 1$, we color $uu_1$ with an available $2$-color to obtain a good coloring. In case $f(u_1u_2), f(u_1u_3) =1, 2_1$. We color $uu_1$ with an available $2$-color if $f(u_4u_5) = 1$, and color $uu_1$ with $f(u_4u_5)$ if $f(u_4u_5)$ is a $2$-color. This is a contradiction.

\textbf{Case 1.2:} $u_3$ is a $2$-vertex. Let $N(u_3) = \{u_1,u_5\}$. 
By symmetry, $f(u_1u_2), f(u_1u_3) = 1, 2_1$ or $2_1, 1$.
We have three subcases depending on the degree of $u_4$.

\textbf{Case 1.2.1:} $u_4$ is a $1$-vertex. In case $f(u_1u_2), f(u_1u_3) = 1, 2_1$, say $f(u_2u_4)=2_2$, we can color $uu_1$ with an available $2$-color to obtain a good coloring, unless $u_5$ is a $3$-vertex and $u_3$ sees all five colors. However, we color $uu_1$ with another available $2$-color in $\{2_3, 2_4\}$ to obtain a good coloring. In case $f(u_1u_2), f(u_1u_3) = 2_1, 1$, we color $uu_1$ with an available $2$-color to obtain a good coloring unless the following two bad cases occur. The first bad case is when $u_5$ is a $2$-vertex and Condition 3 is violated. However, we color $uu_1$ with another available $2$-color in $\{2_3, 2_4\}$ to obtain a good coloring. The second bad case is when $u_3$ sees all five colors. However, we recolor $u_1u_3$ with $f(uu_1)$, $uu_1$ with $1$ and $u_3u_5$ with $1$ to obtain a good coloring.

\textbf{Case 1.2.2:} $u_4$ is a $2$-vertex. Let $N(u_4) = \{u_2,u_6\}$. 

\textbf{Case 1.2.2.1:} $f(u_1u_2), f(u_1u_3) = 1, 2_1$. We may assume $f(u_2u_4) = 2_2$. We color $uu_1$ with an available $2$-color, say $2_3$. We claim $f(u_3u_5)$ is $1$. Suppose not, $f(u_3u_5)$ is $2_2$ or $2_4$, then $f(u_4u_6)$ must be $1$ or $2_4$, since otherwise we obtain a good coloring. If $f(u_4u_6) = 2_4$, then we switch the colors of $uu_1$ and $u_1u_2$, and color $u_2u_4$ with $1$ to obtain a good coloring. This is a contradiction. Thus, $f(u_4u_6) = 1$, we obtain a good coloring unless $u_6$ is a $3$-vertex and Condition 3 is violated at the $2$-thread $u_1u_2u_4u_6$. However, we switch the colors of $uu_1$ and $u_1u_2$ to obtain a good coloring. This is again a contradiction. Hence, $f(u_3u_5)=1$. 

We further claim $f(u_4u_6)$ must be $1$ or $2_4$, since otherwise we obtain a good coloring unless $u_3$ sees all five colors. However, we switch the colors of $u_1u_2$ and $u_2u_4$. We recolor $uu_1$ with $1$ to obtain a good coloring. In case $f(u_4u_6)=2_4$, we recolor $uu_1$ and $u_2u_4$ with $1$, $u_1u_2$ with $2_2$. We obtain a good coloring unless $u_3$ sees all five colors. However, we recolor $u_1u_2$ with $2_3$ to obtain a good coloring. This is a contradiction. In case $f(u_4u_6)=1$, we obtain a good coloring unless $u_3$ sees all five colors, or $u_6$ is a $3$-vertex and Condition 3 is violated. When these two bad cases occur simultaneously, we can recolor $uu_1$ with $2_4$. When only the first bad case occurs, we can color $uu_1$ with $2_4$ to obtain a good coloring unless $u_6$ is a $3$-vertex and Condition 3 is violated. However, we can recolor $u_1u_2$ with $2_4$ and $uu_1$ with $1$ to obtain a good coloring. When only the second bad case occurs, we recolor $uu_1$ with $2_2$. This is a good coloring unless $u_3$ sees all five colors. However, we recolor $uu_1$ with $2_4$ to obtain a good coloring. This is a contradiction.

\textbf{Case 1.2.2.2:} $f(u_1u_2), f(u_1u_3) = 2_1, 1$. We may assume $f(u_3u_5)=2_2$. By Cases 1.2.1 and 1.2.2.1, $u_5$ is a $3$-vertex.
We color $uu_1$ with an available $2$-color, say $2_3$. Then $f(u_2u_4)=1$ or $2_2$ or $2_4$. We obtain a good coloring unless $u_3$ sees all five colors. However, we recolor $uu_1$ with $2_4$ to obtain a good coloring. 

\textbf{Case 1.2.3:} $u_4$ is a $3$-vertex. We may assume $u_5$ is a $3$-vertex, since the cases when $u_5$ is a $1$-vertex and a $2$-vertex are already covered in Case 1.2.1 and 1.2.2. Let $N(u_4) = \{u_2,u_6,u_7\}$ and $N(u_5) = \{u_3,u_8,u_9\}$.
Since $u_4$ and $u_5$ are both $3$-vertices, by symmetry, it suffices to consider $f(u_1u_2), f(u_1u_3) = 1, 2_1$. We may assume $f(u_2u_4)=2_2$. We claim $f(u_3u_5) = 1$. Suppose not, by symmetry, $f(u_3u_5) = 2_2$ or $2_3$. In the former case, by symmetry, $f(u_4u_6), f(u_4u_7) = 1, 2_1$ or $1, 2_3$ or $2_1, 2_3$. We color $uu_1$ with $2_3$ to obtain a good coloring. In the latter case, if $f(u_4u_6), f(u_4u_7) = 1, 2_4$ or $2_1, 2_4$ or $1, 2_1$, we color $uu_1$ with $2_4$ to obtain a good coloring. If $f(u_4u_6), f(u_4u_7) = 2_1, 2_3$, we recolor $u_1u_2$ with $2_2$, $u_2u_4$ with $1$, and then color $uu_1$ with $1$ to obtain a good coloring. If $f(u_4u_6), f(u_4u_7) = 1, 2_3$, we recolor $u_1u_2$ with $2_4$, and then color $uu_1$ with $1$ to obtain a good coloring. This is a contradiction. Hence, $f(u_3u_5) = 1$.

By symmetry, $f(u_5u_8), f(u_5u_9) = 2_2, 2_3$ or $2_3, 2_4$ or $2_2, 2_4$. 
In the former case, we can color $uu_1$ with $2_3$ to obtain a good coloring unless $f(u_4u_6), f(u_4u_7) = 1, 2_4$ (or $2_1, 2_4$). However, we recolor $u_1u_2$ with $2_3$ (and $u_2u_4$ with $1$), and then color $uu_1$ with $1$ to obtain a good coloring. In the middle case, we color $uu_1$ with $2_3$ to obtain a good coloring unless $2_4 \in \{f(u_4u_6), f(u_4u_7)\}$. However, we recolor $uu_1$ with $2_4$ to obtain a good coloring. In the latter case, we color $uu_1$ with $2_4$ to obtain a good coloring unless $2_3 \in \{f(u_4u_6), f(u_4u_7)\}$. However, we recolor $u_1u_2$ with $2_4$, $u_2u_4$ with $1$ if $1 \notin \{f(u_4u_6), f(u_4u_7)\}$, and color $uu_1$ with $1$ to obtain a good coloring. This is a contradiction.

\textbf{Case 1.3:} $u_3$ is a $3$-vertex. Let $N(u_3) = \{u_1,u_5,u_6\}$. By symmetry, $f(u_1u_2), f(u_1u_3) = 1, 2_1$ or $2_1, 1$. We first show $f(u_1u_2), f(u_1u_3) = 1, 2_1$. Suppose not, i.e., $f(u_1u_2), f(u_1u_3) = 2_1, 1$. We claim $u_4$ is a $3$-vertex, since otherwise, we can apply Condition 2 on $u_1$ in $G'$, and color $uu_1$ with an available $2$-color to obtain a good coloring. Thus, $u_4$ is a $3$-vertex. We further claim $f(u_2u_4) = 1$, since otherwise we can color $uu_1$ with an available $2$-color to obtain a good coloring (and no condition will be violated). However, by Condition 3, we know $\{f(u_4u_7), f(u_4u_8)\} \neq \{f(u_3u_5), f(u_3u_6)\}$, and thus we have an available $2$-color to color $uu_1$ to obtain a good coloring (Condition 2 is not violated at $u_2$). This is a contradiction.

Therefore, $f(u_1u_2), f(u_1u_3) = 1, 2_1$. Say $f(u_2u_4) = 2_2$. We claim $u_4$ is not a $1$-vertex, since otherwise, we can apply Condition 2 on $u_1$ in $G'$, and color $uu_1$ with an available $2$-color to obtain a good coloring. 

\textbf{Case 1.3.1:} $u_4$ is a $2$-vertex. Let $N(u_4) = \{u_2,u_7\}$. In case $u_7$ is a $1$-vertex or a $2$-vertex, by Condition 2 on $u_1$ in $G'$, we can color $uu_1$ with an available $2$-color to obtain a good coloring unless $u_2$ sees all five colors. However, since $f(u_4u_7) \neq 1$ and $f(u_4u_7) \neq f(uu_1)$, we switch the colors of $uu_1$ and $u_1u_2$, then we recolor $u_2u_4$ with $1$ to obtain a good coloring. Thus, we assume $u_7$ is a $3$-vertex, say $N(u_7) = \{u_4, u_8, u_9\}$. 

By symmetry, $f(u_4u_7)$ can be $1, 2_1$, or $2_3$. We claim $f(u_4u_7) = 1$. Suppose not, i.e., $f(u_4u_7) = 2_1$ or $2_3$. By Condition 1, $1 \in \{f(u_7u_8), f(u_7u_9)\}$. By applying Condition 2 on $u_1$ in $G'$, we color $uu_1$ with an available $2$-color to obtain a good coloring (Condition 3 is not violated) unless $f(u_4u_7) = 2_3$ and $uu_1$ only has $2_4$ available. However, we switch the colors of $uu_1$ and $u_1u_2$, and recolor $u_2u_4$ with $1$ to obtain a good coloring.

Thus, $f(u_4u_7) = 1$. We claim $2_1 \in \{f(u_7u_8), f(u_7u_9)\}$, say $f(u_7u_8) = 2_1$, since otherwise, by Condition 2, we always have an available $2$-color to color $uu_1$ and obtain a good coloring (Condition 3 is not violated). By symmetry, we may assume $f(u_7u_9) = 2_3$. By Condition 2 again, we always have an available $2$-color, say $2_x$, to use at $uu_1$. We switch the colors of $uu_1$ and $u_1u_2$ to obtain a good coloring. 

\textbf{Case 1.3.2:} $u_4$ is a $3$-vertex. 
By applying Condition 2 on $u_1$ in $G'$, we can color $uu_1$ with an available $2$-color, say $2_3$, to obtain a good coloring unless $u_2$ sees all five colors. However, by applying Condition 2 on $u_2$ in $G'$, we can recolor $u_1u_2$ with $2_3$ (and $u_2u_4$ with $1$ if necessary, in order to keep Condition 1), and then color $uu_1$ with $1$ to obtain a good coloring. This is a contradiction.

\textbf{Case 2:} At least one of $u_2, u_3$ is a $3$-vertex. 
By symmetry, we may assume $u_2$ is a $3$-vertex. By Case 1, $u_3$ cannot be a $2$-vertex. If $u_3$ is a $1$-vertex, by symmetry, $f(u_1u_2), f(u_1u_3) = 2_1, 2_2$ or $1, 2_1$ or $2_1, 1$. In the first case, we can color $uu_1$ with $1$ to obtain a good coloring. In the latter two cases, since $u_1$ sees at most three $2$-colors, we always have an available $2$-color to color $uu_1$ to obtain a good coloring. If $u_3$ is a $3$-vertex, by symmetry, $f(u_1u_2), f(u_1u_3) = 2_1, 2_2$ or $1, 2_1$. In the former case, we can color $uu_1$ with $1$ to obtain a good coloring. In the latter case, by applying Condition 2 on $u_1$ in $G'$, we always have an available $2$-color to color $uu_1$ to obtain a good coloring. This is a contradiction.
\end{proof}

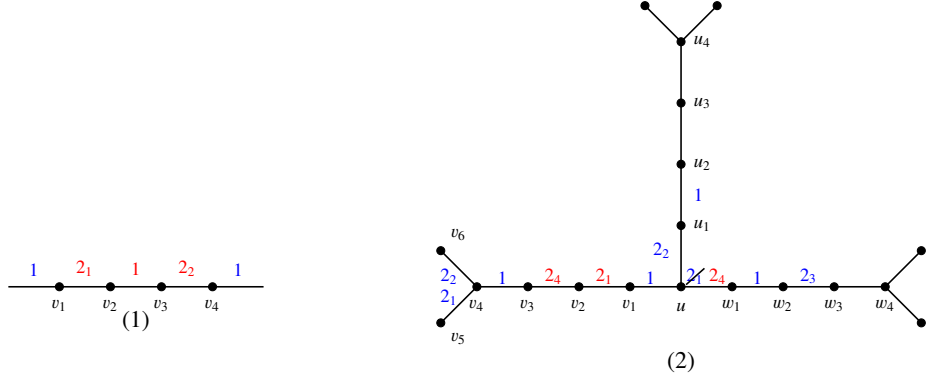
\begin{figure}[ht]
\vspace{-3mm}
 \begin{center}
    \hfill
   \begin{tikzpicture}[
       scale=0.45,
       baseline=0,
       dot/.style={circle, draw=black, fill=black, inner sep=0pt, minimum size=3pt},
        execute at begin picture={\def\scriptsize{\fontsize{7}{8.5}\selectfont}},
       elabel/.style={inner sep=1.5pt}
   ]
       \draw[semithick] (-1.5,0) -- (6.0,0);

       \node[dot] at (0,0) {};
       \node[dot] at (1.5,0) {};
       \node[dot] at (3.0,0) {};
       \node[dot] at (4.5,0) {};

       \node[font=\scriptsize, below=1pt] at (0,0) {$v_1$};
       \node[font=\scriptsize, below=1pt] at (1.5,0) {$v_2$};
       \node[font=\scriptsize, below=1pt] at (3.0,0) {$v_3$};
       \node[font=\scriptsize, below=1pt] at (4.5,0) {$v_4$};
       \node[font=\scriptsize, text=blue, above=.1pt] at (-0.75,0) {$1$};
       \node[font=\scriptsize, text=red, above=.1pt] at (0.75,0) {$2_1$};
       \node[font=\scriptsize, text=red, above=.5pt] at (2.25,0) {$1$};
       \node[font=\scriptsize, text=red, above=.1pt] at (3.75,0) {$2_2$};
       \node[font=\scriptsize, text=blue, above=.1pt] at (5.25,0) {$1$};
       \node[font=\footnotesize] at (2.25,-1.0) {(1)};
   \end{tikzpicture} \hfill
   \begin{tikzpicture}[
       scale=0.45,
       baseline=0,
       dot/.style={circle, draw=black, fill=black, inner sep=0pt, minimum size=3pt},
        execute at begin picture={\def\scriptsize{\fontsize{7}{8.5}\selectfont}},
       elabel/.style={inner sep=1.5pt}
   ]
       \coordinate (h0) at (0,0);
       \coordinate (h1) at (1.5,0);
       \coordinate (h2) at (3.0,0);
       \coordinate (h3) at (4.5,0);
       \coordinate (h4) at (6.0,0);
       \coordinate (h5) at (7.5,0);
       \coordinate (h6) at (9.0,0);
       \coordinate (h7) at (10.5,0);
       \coordinate (h8) at (12.0,0);

       \coordinate (v1) at (6.0,1.8);
       \coordinate (v2) at (6.0,3.6);
       \coordinate (v3) at (6.0,5.4);
       \coordinate (v4) at (6.0,7.2);

       \path (h0) ++(135:1.5) coordinate (left_up);
       \path (h0) ++(225:1.5) coordinate (left_down);

       \path (h8) ++(45:1.5) coordinate (right_up);
       \path (h8) ++(-45:1.5) coordinate (right_down);

       \path (v4) ++(135:1.5) coordinate (top_left);
       \path (v4) ++(45:1.5) coordinate (top_right);

       \draw[semithick] (left_up) -- (h0) -- (left_down);
       \draw[semithick] (h0) -- (h8);
       \draw[semithick] (right_up) -- (h8) -- (right_down);
       \draw[semithick] (h4) -- (v4);
       \draw[semithick] (top_left) -- (v4) -- (top_right);

       \node[dot] at (left_up) {};
       \node[dot] at (left_down) {};
       \node[dot] at (right_up) {};
       \node[dot] at (right_down) {};
       \node[dot] at (top_left) {};
       \node[dot] at (top_right) {};

       \node[dot] at (h0) {};
       \node[dot] at (h1) {};
       \node[dot] at (h2) {};
       \node[dot] at (h3) {};
       \node[dot] at (h4) {};
       \node[dot] at (h5) {};
       \node[dot] at (h6) {};
       \node[dot] at (h7) {};
       \node[dot] at (h8) {};

       \node[dot] at (v1) {};
       \node[dot] at (v2) {};
       \node[dot] at (v3) {};
       \node[dot] at (v4) {};

       \node[font=\scriptsize, above right=0pt] at (left_up) {$v_6$};
       \node[font=\scriptsize, below right=0pt] at (left_down) {$v_5$};
       \node[font=\scriptsize, below=1pt] at (h0) {$v_4$};
       \node[font=\scriptsize, below=1pt] at (h1) {$v_3$};
       \node[font=\scriptsize, below=1pt] at (h2) {$v_2$};
       \node[font=\scriptsize, below=1pt] at (h3) {$v_1$};
       \node[font=\scriptsize, below=2pt] at (h4) {$u$};
       \node[font=\scriptsize, below=1pt] at (h5) {$w_1$};
       \node[font=\scriptsize, below=1pt] at (h6) {$w_2$};
       \node[font=\scriptsize, below=1pt] at (h7) {$w_3$};
       \node[font=\scriptsize, below=1pt] at (h8) {$w_4$};
       \node[font=\scriptsize, right=1pt] at (v1) {$u_1$};
       \node[font=\scriptsize, right=1pt] at (v2) {$u_2$};
       \node[font=\scriptsize, right=1pt] at (v3) {$u_3$};
       \node[font=\scriptsize, right=1pt] at (v4) {$u_4$};
      \node[font=\scriptsize, text=blue, below right=3pt and -4pt] at (left_up) {$2_2$};
\node[font=\scriptsize, text=blue, above right=3pt and -4pt] at (left_down) {$2_1$};
       \path (h0) -- node[font=\scriptsize, text=blue, inner sep=0pt, above=1pt] {$1$} (h1);
       \path (h1) -- node[font=\scriptsize, text=red, inner sep=0pt, above=1pt] {$2_4$} (h2);
       \path (h2) -- node[font=\scriptsize, text=red, inner sep=0pt, above=1pt] {$2_1$} (h3);
       \path (h3) -- node[font=\scriptsize, text=blue, inner sep=0pt, above=.8pt, pos=.4] {$1$} (h4);
       \path (h4) -- coordinate[pos=.50] (figFiveChange) (h5);
       \node[font=\scriptsize, text=blue, inner sep=0pt, anchor=south east, xshift=-.8pt, yshift=.8pt] (figFiveOld) at (figFiveChange) {$2_1$};
       \node[font=\scriptsize, inner sep=0pt, text=red, anchor=south west, xshift=.8pt, yshift=.8pt] at (figFiveChange) {$2_4$};
       \draw[semithick] (figFiveOld.south west) -- (figFiveOld.north east);
       \path (h5) -- node[font=\scriptsize, text=blue, inner sep=0pt, above=1pt] {$1$} (h6);
       \path (h6) -- node[font=\scriptsize, text=blue, inner sep=0pt, above=1pt] {$2_3$} (h7);
       \path (h4) -- node[font=\scriptsize, text=blue, left=.1pt, pos=.60] {$2_2$} (v1);
       \path (v1) -- node[font=\scriptsize, text=blue, right=1pt] {$1$} (v2);
       \node[font=\footnotesize] at (6.0,-2.2) {(2)};
   \end{tikzpicture} 
   \hfill\,
\caption{Lemma~\ref{no 4-chain} and Lemma~\ref{no333}}
\label{lemma 4.3 and 4.4}
\end{center}
\vspace{-8mm}
\end{figure}

\begin{lemma}\label{no 4-chain}
There is no $4$-chain.
\end{lemma}

\begin{proof}
Suppose not, say $v_1v_2v_3v_4$ is a 4-chain. Let $N(v_1) = \{u_1, v_2\}$ and $N(v_4) = \{v_3, u_2\}$. We delete $v_2$ and $v_3$ to obtain a subcubic planar graph $G'$ with girth at least 16. By the minimality of $G$, $G'$ has a good coloring $f$. By symmetry, $f(u_1v_1), f(u_2v_4) = 1, 1$ or $1, 2_1$ or $2_1, 2_1$ or $2_1, 2_2$. We extend $f$ to $G$.

In the first case, regardless of whether $u_1, u_2$ are $2$-vertices or $3$-vertices, each of $v_1v_2, v_3v_4$ sees at most two $2$-colors. Therefore, we always have two available $2$-color to color each of $v_1v_2$ and $v_3v_4$. We color $v_1v_2$ and $v_3v_4$ with $2$-colors in this order, and color $v_2v_3$ with $1$ to obtain a good coloring (see Figure~\ref{lemma 4.3 and 4.4} (1)).
In the second case, regardless of whether $u_2$ is a $2$-vertex or a $3$-vertex, we may assume that $u_2$ sees a $1$-color, since otherwise it reduces to the first case. Since $v_1v_2, v_3v_4$ both see at most two $2$-colors, regardless of whether $u_1$ is a $2$-vertex or a $3$-vertex, we always have two available $2$-colors to color each of $v_1v_2$ and $v_3v_4$. We color $v_1v_2, v_3v_4$ with $2$-colors in this order, and color $v_2v_3$ with $1$ to obtain a good coloring.
In the latter two cases, we may assume that each of $u_1$ and $u_2$ sees a $1$-color, since otherwise it reduces to one of the previous two cases. We can color $v_1v_2$ and $v_3v_4$ with $1$, and color $v_2v_3$ with $2_3$ to obtain a good coloring. This is a contradiction.
\end{proof}

By Lemma \ref{mindegree Thm2} and Lemma \ref{no 4-chain}, any $3$-chain in $G$ must be a $3$-thread.

\begin{lemma}\label{no333}
If a $3$-vertex $u$ is adjacent to two $3$-threads, then it cannot be adjacent to another $3$-thread.
\end{lemma}
\begin{proof}
Suppose not, say $u$ is adjacent to three $3$-threads $uv_1v_2v_3v_4$, $uw_1w_2w_3w_4$, and $uu_1u_2u_3u_4$, where $v_4$, $w_4$, and $u_4$ are $3$-vertices. We delete $v_2$ to obtain a subcubic graph $G'$ with girth at least 16. By the minimality of $G$, $G'$ has a good coloring $f$. Let $N(w_4) = \{w_3, w_5, w_6\}$, $N(u_4) = \{u_3, u_5, u_6\}$, and $N(v_4) = \{v_3, v_5, v_6\}$. By symmetry, $f(uv_1), f(v_3v_4) = 1, 1$ or $1, 2_1$ or $2_1, 1$ or $2_1, 2_1$ or $2_1, 2_2$.

\textbf{Case 1:} $f(uv_1), f(v_3v_4) = 1, 1$. We first prove the following claim, which will be used throughout the remainder of the argument, including Lemmas~\ref{no332},~\ref{no331},~\ref{no322},~\ref{no330},~\ref{no321}, and~\ref{no320}.

\begin{claim}\label{claim 1}
For every good coloring $f$ of $G'$, we have $\{f(uw_1), f(uu_1)\} = \{f(v_4v_5), f(v_4v_6)\}$. Note that the result holds as long as $v_1v_2v_3$ is a $3$-thread. Furthermore, we have $f(u_1u_2)=1$. Note that this holds as long as we have the $3$-thread $v_1v_2v_3$ and both $u_1,u_2$ are $2$-vertices.
\end{claim}

\begin{proof}
Suppose not, i.e., $\{f(uw_1), f(uu_1)\} \neq \{f(v_4v_5), f(v_4v_6)\}$, say $f(uw_1) \notin \{f(v_4v_5), f(v_4v_6)\}$ and $f(v_4v_6) \notin \{f(uw_1), f(uu_1)\}$. We can color $v_1v_2$ with $f(v_4v_6)$ and color $v_2v_3$ with $f(uw_1)$ to obtain a good coloring. Hence, we may assume $\{f(uw_1), f(uu_1)\} = \{f(v_4v_5), f(v_4v_6)\} = \{2_1, 2_2\}$. Note that $u_1$ and $u_2$ are $2$-vertices. 

\textbf{Case (i):} $u_3$ is a $2$-vertex. Suppose $f(u_1u_2)\neq 1$, i.e., $f(u_1u_2) \in \{2_3, 2_4\}$. We may assume that $f(u_1u_2)=2_3$. By Condition 1, we must have $f(u_2u_3) = 1$. We switch the colors of $uv_1$ and $uu_1$, color $v_1v_2$ with $1$ and $v_2v_3$ with $2_3$ to obtain a good coloring. Thus, we have $f(u_1u_2)=1$. 

\textbf{Case (ii):} $u_3$ is a $3$-vertex. Let $N(u_3) = \{u_2, u_4, u_7\}$. Suppose $f(u_1u_2) \neq 1$. By symmetry, we assume $f(u_1u_2) = 2_3$. The proof is identical to Case 1, unless Condition 3 is violated at the $2$-thread $uu_1u_2u_3$. In this case, $f(u_1u_2) = 2_3$, $f(u_2u_3)=1$, and $\{f(u_3u_4), f(u_3u_7)\} = \{2_1, 2_2\}$. Note that $w_1$ can have degree $2$ or $3$.

\textbf{Case (ii.A):} $w_1$ is a $2$-vertex. We claim that $f(w_1w_2) = 1$. Otherwise, if $f(w_1w_2) = 2_3$, we can recolor $uv_1$ with $2_4$ and $uu_1$ with $1$, color $v_1v_2$ with $1$, and $v_2v_3$ with $2_3$ to obtain a good coloring; if $f(w_1w_2) = 2_4$, we can recolor $u_1u_2$ with $2_4$, $uv_1$ with $2_3$ and $uu_1$ with $1$, color $v_1v_2$ with $1$, and $v_2v_3$ with $2_4$ to obtain a good coloring. Thus, $f(w_1w_2) = 1$. If $w_1$ does not see $2_3$, we can recolor $uv_1$ with $2_4$ and $uu_1$ with $1$, color $v_1v_2$ with $1$, and $v_2v_3$ with $2_3$ to obtain a good coloring. Otherwise, $w_1$ sees $2_3$, we can recolor $u_1u_2$ with $2_4$, and then recolor $uv_1$ with $2_3$ and $uu_1$ with $1$, color $v_1v_2$ with $1$, and $v_2v_3$ with $2_4$ to obtain a good coloring.

\textbf{Case (ii.B):} $w_1$ is a $3$-vertex. Let $N(w_1) = \{u, w_2, w_7\}$. We claim $1 \in \{f(w_1w_2), f(w_1w_7)\}$, since otherwise, we can switch the colors of $uv_1$ and $uw_1$, color $v_1v_2$ with $1$, and $v_2v_3$ with $2_3$ to obtain a good coloring. However, if $\{f(w_1w_2), f(w_1w_7)\} = \{1, 2_4\}$, then we can recolor $u_1u_2$ with $2_4$, and then recolor $uv_1$ with $2_3$ and $uu_1$ with $1$, color $v_1v_2$ with $1$, and $v_2v_3$ with $2_4$ to obtain a good coloring; if $\{f(w_1w_2), f(w_1w_7)\} = \{1, 2_3\}$, we can recolor $uv_1$ with $2_4$ and $uu_1$ with $1$, color $v_1v_2$ with $1$, and $v_2v_3$ with $2_3$ to obtain a good coloring.
\end{proof}

By Claim \ref{claim 1}, we may assume without loss of generality that $f(uw_1) = f(v_4v_5) = 2_1$ and $f(uu_1) = f(v_4v_6) = 2_2$. And since $w_1$ and $w_2$ are $2$-vertices, by the same argument as in Claim \ref{claim 1}, $f(w_1w_2) = 1$. Thus we have $f(w_1w_2) = f(u_1u_2) = 1$. We recolor $uw_1$ with a color from $\{2_3, 2_4\} \setminus \{f(w_2w_3)\}$, color $v_1v_2$ with $2_1$, and color $v_2v_3$ with the new $f(uw_1)$ to obtain a good coloring (see Figure~\ref{lemma 4.3 and 4.4} (2)). This completes the proof of Case 1.

\textbf{Case 2:} $f(uv_1) = 1$ and $f(v_3v_4) = 2_1$. We claim that one of $f(v_4v_5)$ and $f(v_4v_6)$ is 1, since otherwise we can recolor $v_3v_4$ with 1, which reduces to Case 1. However, we can color $v_2v_3$ with 1. Since $v_1v_2$ sees at most three $2$-colors, we always have an available $2$-color to color $v_1v_2$ to obtain a good coloring.

\textbf{Case 3:} $f(uv_1) = 2_1$ and $f(v_3v_4) = 1$. We claim that one of $f(uw_1)$ and $f(uu_1)$ is 1, since otherwise we can recolor $uv_1$ with 1, which reduces to Case 1. However, we can color $v_1v_2$ with 1. Since $v_2v_3$ sees at most three $2$-colors, we always have an available $2$-color to color $v_2v_3$ to obtain a good coloring. 

\textbf{Case 4:} $f(uv_1) = f(v_3v_4) = 2_1$, or $\{f(uv_1), f(v_3v_4)\} = \{2_1, 2_2\}$. We claim that one of $f(uw_1)$ and $f(uu_1)$ is 1, and one of $f(v_4v_5)$ and $f(v_4v_6)$ is 1, since otherwise we can recolor $uv_1$ with 1 or recolor $v_3v_4$ with 1, which reduces to one of the previous three cases. However, we can color $v_1v_2$ with 1. Since $v_2v_3$ sees at most three $2$-colors, we always have an available $2$-color to color $v_2v_3$ to obtain a good coloring. This is a contradiction, which completes the proof of Lemma \ref{no333}.
\end{proof}

\begin{figure}[ht]
\vspace{-3mm}
 \begin{center}
    \hfill
   \begin{tikzpicture}[
       scale=0.34,
       baseline=0,
       dot/.style={circle, draw=black, fill=black, inner sep=0pt, minimum size=2.8pt},
        execute at begin picture={\def\scriptsize{\fontsize{6}{7.5}\selectfont}},
       elabel/.style={inner sep=1.5pt}
   ]
       \coordinate (h0) at (0,0);
       \coordinate (h1) at (1.5,0);
       \coordinate (h2) at (3.0,0);
       \coordinate (h3) at (4.5,0);
       \coordinate (h4) at (6.0,0);
       \coordinate (h5) at (7.5,0);
       \coordinate (h6) at (9.0,0);
       \coordinate (h7) at (10.5,0);
       \coordinate (h8) at (12.0,0);

       \coordinate (v1) at (6.0,1.8);
       \coordinate (v2) at (6.0,3.6);
       \coordinate (v3) at (6.0,5.4);

       \path (h0) ++(135:1.5) coordinate (left_up);
       \path (h0) ++(225:1.5) coordinate (left_down);

       \path (h8) ++(45:1.5) coordinate (right_up);
       \path (h8) ++(-45:1.5) coordinate (right_down);

       \path (v3) ++(135:1.5) coordinate (top_left);
       \path (v3) ++(45:1.5) coordinate (top_right);

       \draw[semithick] (left_up) -- (h0) -- (left_down);
       \draw[semithick] (h0) -- (h8);
       \draw[semithick] (right_up) -- (h8) -- (right_down);
       \draw[semithick] (h4) -- (v3);
       \draw[semithick] (top_left) -- (v3) -- (top_right);

       \node[dot] at (left_up) {};
       \node[dot] at (left_down) {};
       \node[dot] at (right_up) {};
       \node[dot] at (right_down) {};
       \node[dot] at (top_left) {};
       \node[dot] at (top_right) {};

       \node[dot] at (h0) {};
       \node[dot] at (h1) {};
       \node[dot] at (h2) {};
       \node[dot] at (h3) {};
       \node[dot] at (h4) {};
       \node[dot] at (h5) {};
       \node[dot] at (h6) {};
       \node[dot] at (h7) {};
       \node[dot] at (h8) {};

       \node[dot] at (v1) {};
       \node[dot] at (v2) {};
       \node[dot] at (v3) {};

       \node[font=\scriptsize, above right=0pt] at (left_up) {$v_6$};
       \node[font=\scriptsize, below right=0pt] at (left_down) {$v_5$};
       \node[font=\scriptsize, below=1pt] at (h0) {$v_4$};
       \node[font=\scriptsize, below=1pt] at (h1) {$v_3$};
       \node[font=\scriptsize, below=1pt] at (h2) {$v_2$};
       \node[font=\scriptsize, below=1pt] at (h3) {$v_1$};
       \node[font=\scriptsize, below=2pt] at (h4) {$u$};
       \node[font=\scriptsize, below=1pt] at (h5) {$w_1$};
       \node[font=\scriptsize, below=1pt] at (h6) {$w_2$};
       \node[font=\scriptsize, below=1pt] at (h7) {$w_3$};
       \node[font=\scriptsize, below=1pt] at (h8) {$w_4$};
       \node[font=\scriptsize, right=1pt] at (v1) {$u_1$};
       \node[font=\scriptsize, right=1pt] at (v2) {$u_2$};
       \node[font=\scriptsize, right=1pt] at (v3) {$u_3$};
       \node[font=\scriptsize, above left=0pt] at (top_left) {$u_5$};
       \node[font=\scriptsize, above right=0pt] at (top_right) {$u_4$};
       \node[font=\scriptsize, above left=0pt] at (right_up) {$w_5$};
       \node[font=\scriptsize, below left=0pt] at (right_down) {$w_6$};
    \path (left_up) -- node[font=\scriptsize, text=blue, below=1.2pt, pos=0.45] {$2_1$} (h0);
\path (left_down) -- node[font=\scriptsize, text=blue, above=1.5pt, pos=0.5] {$2_2$} (h0);
      \path (h0) -- node[font=\scriptsize, text=blue, inner sep=0pt, above=1pt] {$1$} (h1);
\path (h1) -- node[font=\scriptsize, text=red, inner sep=0pt, above=1pt, yshift=-0.5pt] {$2_4$} (h2);
\path (h2) -- node[font=\scriptsize, text=red, inner sep=0pt, above=1pt, yshift=-0.5pt] {$2_1$} (h3);
\path (h3) -- node[font=\scriptsize, text=blue, inner sep=0pt, above=1pt, pos=.4] {$1$} (h4);
       \path (h4) -- coordinate[pos=.60] (figSixAChange) (h5);
       \node[font=\scriptsize, text=blue, inner sep=0pt, anchor=south east, xshift=-.8pt, yshift=.8pt] (figSixAOld) at (figSixAChange) {$2_1$};
       \node[font=\scriptsize, inner sep=0pt, text=red, anchor=south west, xshift=.8pt, yshift=.8pt] at (figSixAChange) {$2_4$};
       \draw[semithick] (figSixAOld.south west) -- (figSixAOld.north east);
       \path (h5) -- node[font=\scriptsize, text=blue, inner sep=0pt, above=1.5pt] {$1$} (h6);
       \path (h6) -- node[font=\scriptsize, text=blue, inner sep=0pt, above=1pt] {$2_3$} (h7);
       \path (h4) -- node[font=\scriptsize, text=blue, left=.1pt, pos=.65] {$2_2$} (v1);
       \path (v1) -- node[font=\scriptsize, text=blue, right=1pt] {$1$} (v2);
       \node[font=\footnotesize] at (6.0,-2.2) {(1)};
   \end{tikzpicture} \hfill
   \begin{tikzpicture}[
       scale=0.34,
       baseline=0,
       dot/.style={circle, draw=black, fill=black, inner sep=0pt, minimum size=2.8pt},
        execute at begin picture={\def\scriptsize{\fontsize{6}{7.5}\selectfont}},
       elabel/.style={inner sep=1.5pt}
   ]
       \coordinate (h0) at (0,0);
       \coordinate (h1) at (1.5,0);
       \coordinate (h2) at (3.0,0);
       \coordinate (h3) at (4.5,0);
       \coordinate (h4) at (6.0,0);
       \coordinate (h5) at (7.5,0);
       \coordinate (h6) at (9.0,0);
       \coordinate (h7) at (10.5,0);
       \coordinate (h8) at (12.0,0);

       \coordinate (v1) at (6.0,1.8);
       \coordinate (v2) at (6.0,3.6);

       \path (h0) ++(135:1.5) coordinate (left_up);
       \path (h0) ++(225:1.5) coordinate (left_down);

       \path (h8) ++(45:1.5) coordinate (right_up);
       \path (h8) ++(-45:1.5) coordinate (right_down);

       \path (v2) ++(135:1.5) coordinate (top_left);
       \path (v2) ++(45:1.5) coordinate (top_right);

       \draw[semithick] (left_up) -- (h0) -- (left_down);
       \draw[semithick] (h0) -- (h8);
       \draw[semithick] (right_up) -- (h8) -- (right_down);
       \draw[semithick] (h4) -- (v2);
       \draw[semithick] (top_left) -- (v2) -- (top_right);

       \node[dot] at (left_up) {};
       \node[dot] at (left_down) {};
       \node[dot] at (right_up) {};
       \node[dot] at (right_down) {};
       \node[dot] at (top_left) {};
       \node[dot] at (top_right) {};

       \node[dot] at (h0) {};
       \node[dot] at (h1) {};
       \node[dot] at (h2) {};
       \node[dot] at (h3) {};
       \node[dot] at (h4) {};
       \node[dot] at (h5) {};
       \node[dot] at (h6) {};
       \node[dot] at (h7) {};
       \node[dot] at (h8) {};

       \node[dot] at (v1) {};
       \node[dot] at (v2) {};

       \node[font=\scriptsize, above right=0pt] at (left_up) {$v_6$};
       \node[font=\scriptsize, below right=0pt] at (left_down) {$v_5$};
       \node[font=\scriptsize, below=1pt] at (h0) {$v_4$};
       \node[font=\scriptsize, below=1pt] at (h1) {$v_3$};
       \node[font=\scriptsize, below=1pt] at (h2) {$v_2$};
       \node[font=\scriptsize, below=1pt] at (h3) {$v_1$};
       \node[font=\scriptsize, below=2pt] at (h4) {$u$};
       \node[font=\scriptsize, below=1pt] at (h5) {$w_1$};
       \node[font=\scriptsize, below=1pt] at (h6) {$w_2$};
       \node[font=\scriptsize, below=1pt] at (h7) {$w_3$};
       \node[font=\scriptsize, below=1pt] at (h8) {$w_4$};
       \node[font=\scriptsize, right=1pt] at (v1) {$u_1$};
       \node[font=\scriptsize, right=1pt] at (v2) {$u_2$};
       \node[font=\scriptsize, above left=0pt] at (top_left) {$u_4$};
       \node[font=\scriptsize, above right=0pt] at (top_right) {$u_3$};
       \node[font=\scriptsize, above left=0pt] at (right_up) {$w_5$};
       \node[font=\scriptsize, below left=0pt] at (right_down) {$w_6$};
     \path (left_up) -- node[font=\scriptsize, text=blue, below=1.2pt, pos=0.45] {$2_1$} (h0);
\path (left_down) -- node[font=\scriptsize, text=blue, above=1.5pt, pos=0.5] {$2_2$} (h0);
       \path (h0) -- node[font=\scriptsize, text=blue, inner sep=0pt, above=1pt] {$1$} (h1);
       \path (h1) -- node[font=\scriptsize, text=red, inner sep=0pt, above=.5pt] {$2_3$} (h2);
       \path (h2) -- node[font=\scriptsize, text=red, inner sep=0pt, above=1.5pt] {$1$} (h3);
       \path (h3) -- coordinate[pos=.40] (figSixBChangeLeft) (h4);
       \node[font=\scriptsize, text=blue, inner sep=0pt, anchor=south east, xshift=-.8pt, yshift=.8pt] (figSixBOldLeft) at (figSixBChangeLeft) {$1$};
       \node[font=\scriptsize, inner sep=0pt, text=red, anchor=south west, xshift=.8pt, yshift=.8pt] at (figSixBChangeLeft) {$2_4$};
       \draw[semithick] (figSixBOldLeft.south west) -- (figSixBOldLeft.north east);
       \path (h4) -- node[font=\scriptsize, text=blue, inner sep=0pt, above=.8pt, pos=.62] {$2_1$} (h5);
       \path (h5) -- node[font=\scriptsize, text=blue, inner sep=0pt, above=1.5pt] {$1$} (h6);
       \path (h4) -- coordinate[pos=.64] (figSixBChangeUp) (v1);
       \node[font=\scriptsize, text=blue, inner sep=0pt, anchor=east, xshift=-1pt] (figSixBOldUp) at (figSixBChangeUp) {$2_2$};
       \node[font=\scriptsize, inner sep=0pt, text=red, anchor=west, xshift=1pt] at (figSixBChangeUp) {$1$};
       \draw[semithick] (figSixBOldUp.south west) -- (figSixBOldUp.north east);
       \path (v1) -- node[font=\scriptsize, text=blue, right=1pt] {$2_3$} (v2);
       \node[font=\footnotesize] at (6.0,-2.2) {(2)};
   \end{tikzpicture} \hfill
   \begin{tikzpicture}[
       scale=0.34,
       baseline=0,
       dot/.style={circle, draw=black, fill=black, inner sep=0pt, minimum size=2.8pt},
        execute at begin picture={\def\scriptsize{\fontsize{6}{7.5}\selectfont}},
       elabel/.style={inner sep=1.5pt}
   ]
       \coordinate (h0) at (0,0);
       \coordinate (h1) at (1.5,0);
       \coordinate (h2) at (3.0,0);
       \coordinate (h3) at (4.5,0);
       \coordinate (h4) at (6.0,0);
       \coordinate (h5) at (7.5,0);
       \coordinate (h6) at (9.0,0);
       \coordinate (h7) at (10.5,0);
       \coordinate (h8) at (12.0,0);

       \coordinate (v1) at (6.0,1.8);

       \path (h0) ++(135:1.5) coordinate (left_up);
       \path (h0) ++(225:1.5) coordinate (left_down);

       \path (h8) ++(45:1.5) coordinate (right_up);
       \path (h8) ++(-45:1.5) coordinate (right_down);

       \path (v1) ++(135:1.5) coordinate (top_left);
       \path (v1) ++(45:1.5) coordinate (top_right);

       \draw[semithick] (left_up) -- (h0) -- (left_down);
       \draw[semithick] (h0) -- (h8);
       \draw[semithick] (right_up) -- (h8) -- (right_down);
       \draw[semithick] (h4) -- (v1);
       \draw[semithick] (top_left) -- (v1) -- (top_right);

       \node[dot] at (left_up) {};
       \node[dot] at (left_down) {};
       \node[dot] at (right_up) {};
       \node[dot] at (right_down) {};
       \node[dot] at (top_left) {};
       \node[dot] at (top_right) {};

       \node[dot] at (h0) {};
       \node[dot] at (h1) {};
       \node[dot] at (h2) {};
       \node[dot] at (h3) {};
       \node[dot] at (h4) {};
       \node[dot] at (h5) {};
       \node[dot] at (h6) {};
       \node[dot] at (h7) {};
       \node[dot] at (h8) {};

       \node[dot] at (v1) {};

       \node[font=\scriptsize, above right=0pt] at (left_up) {$v_6$};
       \node[font=\scriptsize, below right=0pt] at (left_down) {$v_5$};
       \node[font=\scriptsize, below=1pt] at (h0) {$v_4$};
       \node[font=\scriptsize, below=1pt] at (h1) {$v_3$};
       \node[font=\scriptsize, below=1pt] at (h2) {$v_2$};
       \node[font=\scriptsize, below=1pt] at (h3) {$v_1$};
       \node[font=\scriptsize, below=2pt] at (h4) {$u$};
       \node[font=\scriptsize, below=1pt] at (h5) {$w_1$};
       \node[font=\scriptsize, below=1pt] at (h6) {$w_2$};
       \node[font=\scriptsize, below=1pt] at (h7) {$w_3$};
       \node[font=\scriptsize, below=1pt] at (h8) {$w_4$};
       \node[font=\scriptsize, inner sep=0pt, right=1.5pt] at (v1) {$u_1$};
       \node[font=\scriptsize, above left=0pt] at (top_left) {$u_3$};
       \node[font=\scriptsize, above right=0pt] at (top_right) {$u_2$};
       \node[font=\scriptsize, above left=0pt] at (right_up) {$w_5$};
       \node[font=\scriptsize, below left=0pt] at (right_down) {$w_6$};
        \path (left_up) -- node[font=\scriptsize, text=blue, below=1.2pt, pos=0.45] {$2_1$} (h0);
\path (left_down) -- node[font=\scriptsize, text=blue, above=1.5pt, pos=0.5] {$2_2$} (h0);
       \path (h0) -- node[font=\scriptsize, text=blue, inner sep=0pt, above=1pt] {$1$} (h1);
       \path (h1) -- node[font=\scriptsize, text=red, inner sep=0pt, above=.5pt] {$2_3$} (h2);
       \path (h2) -- node[font=\scriptsize, text=red, inner sep=0pt, above=1pt] {$1$} (h3);
       \path (h3) -- coordinate[pos=.40] (figSixCChangeLeft) (h4);
       \node[font=\scriptsize, text=blue, inner sep=0pt, anchor=south east, xshift=-.8pt, yshift=.8pt] (figSixCOldLeft) at (figSixCChangeLeft) {$1$};
       \node[font=\scriptsize, inner sep=0pt, text=red, anchor=south west, xshift=.8pt, yshift=.8pt] at (figSixCChangeLeft) {$2_2$};
       \draw[semithick] (figSixCOldLeft.south west) -- (figSixCOldLeft.north east);
       \path (h4) -- node[font=\scriptsize, text=blue, inner sep=0pt, above=.8pt, pos=.62] {$2_1$} (h5);
       \path (h5) -- node[font=\scriptsize, text=blue, inner sep=0pt, above=1.5pt] {$1$} (h6);
       \path (h4) -- coordinate[pos=.64] (figSixCChangeUp) (v1);
       \node[font=\scriptsize, text=blue, inner sep=0pt, anchor=east, xshift=-1pt] (figSixCOldUp) at (figSixCChangeUp) {$2_2$};
       \node[font=\scriptsize, inner sep=0pt, text=red, anchor=west, xshift=1pt] at (figSixCChangeUp) {$1$};
       \draw[semithick] (figSixCOldUp.south west) -- (figSixCOldUp.north east);
       \path (top_left) -- node[font=\scriptsize, text=blue, fill=white, inner sep=.3pt, below left=1.2pt, pos=.38] {$2_4$} (v1);
       \path (top_right) -- node[font=\scriptsize, text=blue, fill=white, inner sep=.3pt, below right=1.2pt, pos=.38] {$2_3$} (v1);
       \node[font=\footnotesize] at (6.0,-2.2) {(3)};
   \end{tikzpicture} 
   \hfill\,
\caption{Lemma~\ref{no332}, Lemma~\ref{no331} and Lemma~\ref{no330}}
\label{lemma4.6, 4.7 and 4.8}
\end{center}
\vspace{-8mm}
\end{figure}
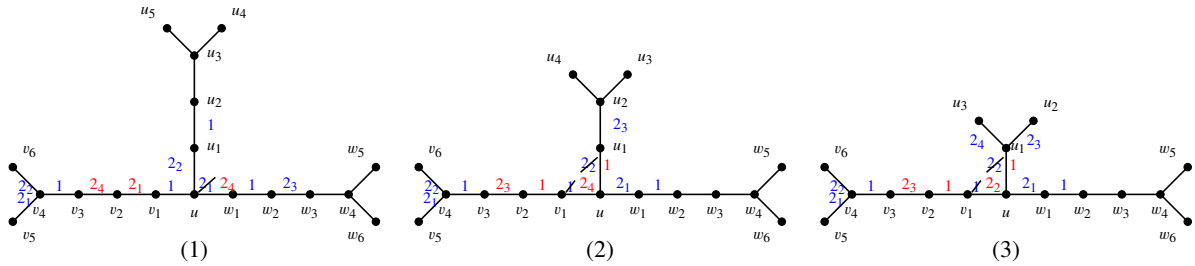

\begin{lemma}\label{no332}
If a $3$-vertex $u$ is adjacent to two $3$-threads, then it cannot be adjacent to a $2$-thread.
\end{lemma}

\begin{proof}
Suppose not, say $u$ is adjacent to two $3$-threads $uv_1v_2v_3v_4$ and $uw_1w_2w_3w_4$, and a $2$-thread $uu_1u_2u_3$, where $v_4, w_4$, and $u_3$ are $3$-vertices. We delete $v_2$ to obtain a subcubic planar graph $G'$ with girth at least $16$. By the minimality of $G$, $G'$ has a good coloring $f$. Let $N(u_3) = \{u_2, u_4, u_5\}$ and $N(v_4) = \{v_3, v_5, v_6\}$. By symmetry, $f(uv_1), f(v_3v_4) = 1, 1$ or $1, 2_1$ or $2_1, 1$ or $2_1, 2_1$ or $2_1, 2_2$.

In case $f(uv_1), f(v_3v_4) = 1, 1$, by Claim \ref{claim 1}, we may assume without loss of generality that $f(uw_1) = f(v_4v_5) = 2_1$ and $f(uu_1) = f(v_4v_6) = 2_2$. Furthermore, it holds that $f(u_1u_2) = 1$ and $f(w_1w_2) = 1$. We can recolor $uw_1$ with a color from $\{2_3, 2_4\} \setminus \{f(w_2w_3)\}$. Then we can color $v_1v_2$ with $2_1$ and color $v_2v_3$ with the new $f(uw_1)$ to obtain a good coloring (see Figure~\ref{lemma4.6, 4.7 and 4.8} (1)).

The proofs for the remaining four cases are identical to that of Lemma \ref{no333}. This is a contradiction.
\end{proof}

\begin{lemma}\label{no331}
If a $3$-vertex $u$ is adjacent to two $3$-threads, then it cannot be adjacent to a $1$-thread.
\end{lemma}

\begin{proof}

Suppose not, say $u$ is adjacent to two $3$-threads $uv_1v_2v_3v_4$ and $uw_1w_2w_3w_4$, and a $1$-thread $uu_1u_2$, where $v_4, w_4$, and $u_2$ are $3$-vertices. We delete $v_2$ to obtain a subcubic planar graph $G'$ with girth at least $16$. By the minimality of $G$, $G'$ has a good coloring $f$. Let $N(u_2) = \{u_1, u_3, u_4\}$ and $N(v_4) = \{v_3, v_5, v_6\}$. By symmetry, $f(uv_1), f(v_3v_4) = 1, 1$ or $1, 2_1$ or $2_1, 1$ or $2_1, 2_1$ or $2_1, 2_2$.

In case $f(uv_1), f(v_3v_4) = 1, 1$, by Claim \ref{claim 1}, we may assume without loss of generality that $f(uw_1) = f(v_4v_5) = 2_1$ and $f(uu_1) = f(v_4v_6) = 2_2$. Furthermore, it holds that $f(w_1w_2) = 1$. 

If $f(u_1u_2)$ is a $2$-color, say $2_3$, we can recolor $uv_1$ with $2_4$ and recolor $uu_1$ with $1$. Subsequently, we can color $v_1v_2$ with 1 and $v_2v_3$ with $2_3$ to obtain a good coloring (see Figure~\ref{lemma4.6, 4.7 and 4.8} (2)). Next, if $f(u_1u_2) = 1$, by Condition 2, one of $f(u_2u_3)$ and $f(u_2u_4)$ is $2_1$, say $f(u_2u_3) = 2_1$. we can recolor $uu_1$ with a color from $\{2_3, 2_4\} \setminus \{f(u_2u_4)\}$, and then color $v_1v_2$ with $2_2$ and $v_2v_3$ with the new $f(uu_1)$ to obtain a good coloring. 

The proofs for the remaining four cases are identical to that of Lemma \ref{no333}. This is a contradiction.
\end{proof}

\begin{lemma}\label{no330}
A $3$-vertex $u$ cannot be adjacent to two $3$-threads.
\end{lemma}

\begin{proof}
Suppose not, say $u$ is adjacent to two $3$-threads $uv_1v_2v_3v_4$ and $uw_1w_2w_3w_4$, and another vertex $u_1$, where $v_4, w_4$ are $3$-vertices. By Lemmas~\ref{no333},~\ref{no332}, and~\ref{no331}, $u_1$ must be a $3$-vertex. We delete $v_2$ to obtain a subcubic planar graph $G'$ with girth at least $16$. By the minimality of $G$, $G'$ has a good coloring $f$. Let $N(w_4) = \{w_3, w_5, w_6\}, N(u_1) = \{u, u_2, u_3\}$, and $N(v_4) = \{v_3, v_5, v_6\}$. By symmetry, $f(uv_1), f(v_3v_4) = 1, 1$ or $1, 2_1$ or $2_1, 1$ or $2_1, 2_1$ or $2_1, 2_2$.

In the case $f(uv_1), f(v_3v_4) = 1, 1$, by Claim \ref{claim 1}, we may assume without loss of generality that $f(uw_1) = f(v_4v_5) = 2_1$ and $f(uu_1) = f(v_4v_6) = 2_2$. Furthermore, it holds that $f(w_1w_2) = 1$.

We may assume that $1 \in \{f(u_1u_2), f(u_1u_3)\}$, since otherwise $\{f(u_1u_2), f(u_1u_3)\} = \{2_3, 2_4\}$, and we can switch the colors of $uu_1$ and $uv_1$, color $v_1v_2$ with $1$, and $v_2v_3$ with $2_3$ to obtain a good coloring (see Figure~\ref{lemma4.6, 4.7 and 4.8} (3)). This is a contradiction. Therefore, we may assume by symmetry that $\{f(u_1u_2), f(u_1u_3)\} = \{1,2_3\}$. Furthermore, we claim $f(w_2w_3) = 2_4$, since otherwise, we can recolor $uw_1$ with $2_4$, color $v_1v_2$ with $2_1$, and $v_2v_3$ with $2_4$ to obtain a good coloring. 

We must have $f(w_3w_4)$ is a $2$-color, since otherwise, $f(w_3w_4) = 1$, we can recolor $uw_1$ with $2_4$ and $w_2w_3$ with a color from $\{2_1, 2_2, 2_3\} \setminus \{f(w_4w_5), f(w_4w_6)\}$. We can then color $v_1v_2$ with $2_1$ and $v_2v_3$ with $2_4$ to obtain a good coloring.
Thus, $f(w_3w_4)$ is a $2$-color. By Condition 1, one of $f(w_4w_5)$ and $f(w_4w_6)$ is 1. In this case, we can recolor $w_2w_3$ with 1, $w_1w_2$ with $2_4$, $uw_1$ with 1, and $uv_1$ with $2_1$. Subsequently, we can color $v_1v_2$ with $1$ and $v_2v_3$ with $2_3$ to obtain a good coloring. 

The proofs for the remaining four cases are identical to that of Lemma \ref{no333}. This is a contradiction.
\end{proof}

\begin{figure}[ht]
\vspace{-3mm}
 \begin{center}
    \hfill
   \begin{tikzpicture}[
       scale=0.36,
       baseline=0,
       dot/.style={circle, draw=black, fill=black, inner sep=0pt, minimum size=2.8pt},
        execute at begin picture={\def\scriptsize{\fontsize{6}{7.5}\selectfont}},
       elabel/.style={inner sep=1.5pt}
   ]
       \coordinate (h0) at (0,0);
       \coordinate (h1) at (1.5,0);
       \coordinate (h2) at (3.0,0);
       \coordinate (h3) at (4.5,0);
       \coordinate (h4) at (6.0,0);
       \coordinate (h5) at (7.5,0);
       \coordinate (h6) at (9.0,0);
       \coordinate (h7) at (10.5,0);

       \coordinate (v1) at (6.0,1.8);
       \coordinate (v2) at (6.0,3.6);
       \coordinate (v3) at (6.0,5.4);

       \path (h0) ++(135:1.5) coordinate (left_up);
       \path (h0) ++(225:1.5) coordinate (left_down);

       \path (h7) ++(45:1.5) coordinate (right_up);
       \path (h7) ++(-45:1.5) coordinate (right_down);

       \path (v3) ++(135:1.5) coordinate (top_left);
       \path (v3) ++(45:1.5) coordinate (top_right);

       \draw[semithick] (left_up) -- (h0) -- (left_down);
       \draw[semithick] (h0) -- (h7);
       \draw[semithick] (right_up) -- (h7) -- (right_down);
       \draw[semithick] (h4) -- (v3);
       \draw[semithick] (top_left) -- (v3) -- (top_right);

       \node[dot] at (left_up) {};
       \node[dot] at (left_down) {};
       \node[dot] at (right_up) {};
       \node[dot] at (right_down) {};
       \node[dot] at (top_left) {};
       \node[dot] at (top_right) {};

       \node[dot] at (h0) {};
       \node[dot] at (h1) {};
       \node[dot] at (h2) {};
       \node[dot] at (h3) {};
       \node[dot] at (h4) {};
       \node[dot] at (h5) {};
       \node[dot] at (h6) {};
       \node[dot] at (h7) {};

       \node[dot] at (v1) {};
       \node[dot] at (v2) {};
       \node[dot] at (v3) {};

       \node[font=\scriptsize, above right=0pt] at (left_up) {$v_6$};
       \node[font=\scriptsize, below right=0pt] at (left_down) {$v_5$};
       \node[font=\scriptsize, below=1pt] at (h0) {$v_4$};
       \node[font=\scriptsize, below=1pt] at (h1) {$v_3$};
       \node[font=\scriptsize, below=1pt] at (h2) {$v_2$};
       \node[font=\scriptsize, below=1pt] at (h3) {$v_1$};
       \node[font=\scriptsize, below=2pt] at (h4) {$u$};
       \node[font=\scriptsize, below=1pt] at (h5) {$w_1$};
       \node[font=\scriptsize, below=1pt] at (h6) {$w_2$};
       \node[font=\scriptsize, below=1pt] at (h7) {$w_3$};
       \node[font=\scriptsize, right=1pt] at (v1) {$u_1$};
       \node[font=\scriptsize, right=1pt] at (v2) {$u_2$};
       \node[font=\scriptsize, right=1pt] at (v3) {$u_3$};
       \node[font=\scriptsize, above left=0pt] at (top_left) {$u_5$};
       \node[font=\scriptsize, above right=0pt] at (top_right) {$u_4$};
       \node[font=\scriptsize, above left=0pt] at (right_up) {$w_5$};
       \node[font=\scriptsize, below left=0pt] at (right_down) {$w_4$};
 \path (left_up) -- node[font=\scriptsize, text=blue, below=4.5pt, pos=0.30] {$2_1$} (h0);
\path (left_down) -- node[font=\scriptsize, text=blue, above=4.5pt, pos=0.30] {$2_2$} (h0);
       \path (h0) -- node[font=\scriptsize, text=blue, inner sep=0pt, above=1pt] {$1$} (h1);
       \path (h1) -- node[font=\scriptsize, text=red, inner sep=0pt, above=1pt] {$2_4$} (h2);
       \path (h2) -- node[font=\scriptsize, text=red, inner sep=0pt, above=1pt] {$2_1$} (h3);
       \path (h3) -- node[font=\scriptsize, text=blue, inner sep=0pt, above=.8pt, pos=.38] {$1$} (h4);
       \path (h4) -- coordinate[pos=.60] (figSevenAChange) (h5);
       \node[font=\scriptsize, text=blue, inner sep=0pt, anchor=south east, xshift=-.8pt, yshift=.8pt] (figSevenAOld) at (figSevenAChange) {$2_1$};
       \node[font=\scriptsize, inner sep=0pt, text=red, anchor=south west, xshift=.8pt, yshift=.8pt] at (figSevenAChange) {$2_4$};
       \draw[semithick] (figSevenAOld.south west) -- (figSevenAOld.north east);
       \path (h5) -- node[font=\scriptsize, text=blue, inner sep=0pt, above=1pt] {$1$} (h6);
       \path (h6) -- node[font=\scriptsize, text=blue, inner sep=0pt, above=1pt] {$2_3$} (h7);
       \path (h4) -- node[font=\scriptsize, text=blue, inner sep=0pt, left=1pt, pos=.60] {$2_2$} (v1);
       \path (v1) -- node[font=\scriptsize, text=blue, right=1pt] {$1$} (v2);
       \node[font=\footnotesize] at (5.25,-2.2) {(1)};
   \end{tikzpicture} \hfill
   \begin{tikzpicture}[
       scale=0.36,
       baseline=0,
       dot/.style={circle, draw=black, fill=black, inner sep=0pt, minimum size=2.8pt},
        execute at begin picture={\def\scriptsize{\fontsize{6}{7.3}\selectfont}},
       elabel/.style={inner sep=1.5pt}
   ]
       \coordinate (h0) at (0,0);
       \coordinate (h1) at (1.5,0);
       \coordinate (h2) at (3.0,0);
       \coordinate (h3) at (4.5,0);
       \coordinate (h4) at (6.0,0);
       \coordinate (h5) at (7.5,0);
       \coordinate (h6) at (9.0,0);
       \coordinate (h7) at (10.5,0);

       \coordinate (v1) at (6.0,1.8);
       \coordinate (v2) at (6.0,3.6);

       \path (h0) ++(135:1.5) coordinate (left_up);
       \path (h0) ++(225:1.5) coordinate (left_down);

       \path (h7) ++(45:1.5) coordinate (right_up);
       \path (h7) ++(-45:1.5) coordinate (right_down);

       \path (v2) ++(135:1.5) coordinate (top_left);
       \path (v2) ++(45:1.5) coordinate (top_right);

       \draw[semithick] (left_up) -- (h0) -- (left_down);
       \draw[semithick] (h0) -- (h7);
       \draw[semithick] (right_up) -- (h7) -- (right_down);
       \draw[semithick] (h4) -- (v2);
       \draw[semithick] (top_left) -- (v2) -- (top_right);

       \node[dot] at (left_up) {};
       \node[dot] at (left_down) {};
       \node[dot] at (right_up) {};
       \node[dot] at (right_down) {};
       \node[dot] at (top_left) {};
       \node[dot] at (top_right) {};

       \node[dot] at (h0) {};
       \node[dot] at (h1) {};
       \node[dot] at (h2) {};
       \node[dot] at (h3) {};
       \node[dot] at (h4) {};
       \node[dot] at (h5) {};
       \node[dot] at (h6) {};
       \node[dot] at (h7) {};

       \node[dot] at (v1) {};
       \node[dot] at (v2) {};

       \node[font=\scriptsize, above right=0pt] at (left_up) {$v_5$};
       \node[font=\scriptsize, below right=0pt] at (left_down) {$v_6$};
       \node[font=\scriptsize, below=1pt] at (h0) {$v_4$};
       \node[font=\scriptsize, below=1pt] at (h1) {$v_3$};
       \node[font=\scriptsize, below=1pt] at (h2) {$v_2$};
       \node[font=\scriptsize, below=1pt] at (h3) {$v_1$};
       \node[font=\scriptsize, below=2pt] at (h4) {$u$};
       \node[font=\scriptsize, below=1pt] at (h5) {$u_1$};
       \node[font=\scriptsize, below=1pt] at (h6) {$u_2$};
       \node[font=\scriptsize, below=1pt] at (h7) {$u_3$};
       \node[font=\scriptsize, right=1pt] at (v1) {$w_1$};
       \node[font=\scriptsize, inner sep=0pt, right=1.5pt] at (v2) {$w_2$};
       \node[font=\scriptsize, above left=0pt] at (top_left) {$w_4$};
       \node[font=\scriptsize, above right=0pt] at (top_right) {$w_3$};
       \node[font=\scriptsize, above left=0pt] at (right_up) {$u_4$};
       \node[font=\scriptsize, below left=0pt] at (right_down) {$u_5$};
 \path (left_up) -- node[font=\scriptsize, text=blue, below=4.5pt, pos=0.30] {$2_2$} (h0);
\path (left_down) -- node[font=\scriptsize, text=blue, above=4.5pt, pos=0.30] {$2_1$} (h0);
       \path (h0) -- node[font=\scriptsize, text=blue, inner sep=0pt, above=1.5pt] {$1$} (h1);
       \path (h1) -- node[font=\scriptsize, text=red, inner sep=0pt, above=1pt] {$2_3$} (h2);
       \path (h2) -- node[font=\scriptsize, text=red, inner sep=0pt, above=1.5pt] {$1$} (h3);
       \path (h3) -- coordinate[pos=.40] (figSevenBChangeLeft) (h4);
       \node[font=\scriptsize, text=blue, inner sep=0pt, anchor=south east, xshift=-.8pt, yshift=1pt] (figSevenBOldLeft) at (figSevenBChangeLeft) {$1$};
       \node[font=\scriptsize, inner sep=0pt, text=red, anchor=south west, xshift=.8pt, yshift=.8pt] at (figSevenBChangeLeft) {$2_4$};
       \draw[semithick] (figSevenBOldLeft.south west) -- (figSevenBOldLeft.north east);
       \path (h4) -- node[font=\scriptsize, text=blue, inner sep=0pt, above=.8pt, pos=.62] {$2_2$} (h5);
       \path (h5) -- node[font=\scriptsize, text=blue, inner sep=0pt, above=1pt] {$1$} (h6);
       \path (h4) -- coordinate[pos=.64] (figSevenBChangeUp) (v1);
       \node[font=\scriptsize, text=blue, inner sep=0pt, anchor=east, xshift=-1pt] (figSevenBOldUp) at (figSevenBChangeUp) {$2_1$};
       \node[font=\scriptsize, inner sep=0pt, text=red, anchor=west, xshift=1pt] at (figSevenBChangeUp) {$1$};
       \draw[semithick] (figSevenBOldUp.south west) -- (figSevenBOldUp.north east);
       \path (v1) -- node[font=\scriptsize, text=blue, right=1pt] {$2_3$} (v2);
       \path (top_left) -- node[font=\scriptsize, text=blue, fill=white, inner sep=.3pt, below left=1.2pt, pos=.38] {$2_4$} (v2);
       \path (top_right) -- node[font=\scriptsize, text=blue, fill=white, inner sep=.3pt, below right=1.2pt, pos=.38] {$1$} (v2);
       \node[font=\footnotesize] at (5.25,-2.2) {(2)};
   \end{tikzpicture} \hfill
   \begin{tikzpicture}[
       scale=0.36,
       baseline=0,
       dot/.style={circle, draw=black, fill=black, inner sep=0pt, minimum size=2.8pt},
        execute at begin picture={\def\scriptsize{\fontsize{6}{7.5}\selectfont}},
       elabel/.style={inner sep=1.5pt}
   ]
       \coordinate (h0) at (0,0);
       \coordinate (h1) at (1.5,0);
       \coordinate (h2) at (3.0,0);
       \coordinate (h3) at (4.5,0);
       \coordinate (h4) at (6.0,0);
       \coordinate (h5) at (7.5,0);
       \coordinate (h6) at (9.0,0);
       \coordinate (h7) at (10.5,0);

       \coordinate (v1) at (6.0,1.8);

       \path (h0) ++(135:1.5) coordinate (left_up);
       \path (h0) ++(225:1.5) coordinate (left_down);

       \path (h7) ++(45:1.5) coordinate (right_up);
       \path (h7) ++(-45:1.5) coordinate (right_down);

       \path (v1) ++(135:1.5) coordinate (top_left);
       \path (v1) ++(45:1.5) coordinate (top_right);

       \draw[semithick] (left_up) -- (h0) -- (left_down);
       \draw[semithick] (h0) -- (h7);
       \draw[semithick] (right_up) -- (h7) -- (right_down);
       \draw[semithick] (h4) -- (v1);
       \draw[semithick] (top_left) -- (v1) -- (top_right);

       \node[dot] at (left_up) {};
       \node[dot] at (left_down) {};
       \node[dot] at (right_up) {};
       \node[dot] at (right_down) {};
       \node[dot] at (top_left) {};
       \node[dot] at (top_right) {};

       \node[dot] at (h0) {};
       \node[dot] at (h1) {};
       \node[dot] at (h2) {};
       \node[dot] at (h3) {};
       \node[dot] at (h4) {};
       \node[dot] at (h5) {};
       \node[dot] at (h6) {};
       \node[dot] at (h7) {};

       \node[dot] at (v1) {};

       \node[font=\scriptsize, above right=0pt] at (left_up) {$v_5$};
       \node[font=\scriptsize, below right=0pt] at (left_down) {$v_6$};
       \node[font=\scriptsize, below=1pt] at (h0) {$v_4$};
       \node[font=\scriptsize, below=1pt] at (h1) {$v_3$};
       \node[font=\scriptsize, below=1pt] at (h2) {$v_2$};
       \node[font=\scriptsize, below=1pt] at (h3) {$v_1$};
       \node[font=\scriptsize, below=2pt] at (h4) {$u$};
       \node[font=\scriptsize, below=1pt] at (h5) {$u_1$};
       \node[font=\scriptsize, below=1pt] at (h6) {$u_2$};
       \node[font=\scriptsize, below=1pt] at (h7) {$u_3$};
       \node[font=\scriptsize, inner sep=0pt,  right=1.5pt] at (v1) {$w_1$};
       \node[font=\scriptsize, above left=0pt] at (top_left) {$w_3$};
       \node[font=\scriptsize, above right=0pt] at (top_right) {$w_2$};
       \node[font=\scriptsize, above left=0pt] at (right_up) {$u_4$};
       \node[font=\scriptsize, below left=0pt] at (right_down) {$u_5$};
 \path (left_up) -- node[font=\scriptsize, text=blue, below=4.5pt, pos=0.30] {$2_2$} (h0);
\path (left_down) -- node[font=\scriptsize, text=blue, above=4.5pt, pos=0.30] {$2_1$} (h0);
       \path (h0) -- node[font=\scriptsize, text=blue, inner sep=0pt,  above=1pt] {$1$} (h1);
       \path (h1) -- node[font=\scriptsize, text=red, inner sep=0pt,  above=1pt] {$2_3$} (h2);
       \path (h2) -- node[font=\scriptsize, text=red, inner sep=0pt,  above=1pt] {$1$} (h3);
       \path (h3) -- coordinate[pos=.40] (figSevenCChangeLeft) (h4);
       \node[font=\scriptsize, text=blue, inner sep=0pt, anchor=south east, xshift=-.8pt, yshift=1pt] (figSevenCOldLeft) at (figSevenCChangeLeft) {$1$};
       \node[font=\scriptsize, inner sep=0pt, text=red, anchor=south west, xshift=.8pt, yshift=.8pt] at (figSevenCChangeLeft) {$2_1$};
       \draw[semithick] (figSevenCOldLeft.south west) -- (figSevenCOldLeft.north east);
       \path (h4) -- node[font=\scriptsize, text=blue, inner sep=0pt,  above=.8pt, pos=.62] {$2_2$} (h5);
       \path (h5) -- node[font=\scriptsize, text=blue, inner sep=0pt,  above=1pt] {$1$} (h6);
       \path (h4) -- coordinate[pos=.64] (figSevenCChangeUp) (v1);
       \node[font=\scriptsize, text=blue, inner sep=0pt, anchor=east, xshift=-1pt] (figSevenCOldUp) at (figSevenCChangeUp) {$2_1$};
       \node[font=\scriptsize, inner sep=0pt, text=red, anchor=west, xshift=1pt] at (figSevenCChangeUp) {$1$};
       \draw[semithick] (figSevenCOldUp.south west) -- (figSevenCOldUp.north east);
       \path (top_left) -- node[font=\scriptsize, text=blue, fill=white, inner sep=.3pt, below left=1.2pt, pos=.38] {$2_4$} (v1);
       \path (top_right) -- node[font=\scriptsize, text=blue, fill=white, inner sep=.3pt, below right=1.2pt, pos=.38] {$2_3$} (v1);
       \node[font=\footnotesize] at (5.25,-2.2) {(3)};
   \end{tikzpicture} 
   \hfill\,
\caption{Lemma~\ref{no322}, Lemma~\ref{no321} and Lemma~\ref{no320}}
\label{lemma 4.9, 4.10 and 4.11}
\end{center}
\vspace{-8mm}
\end{figure}
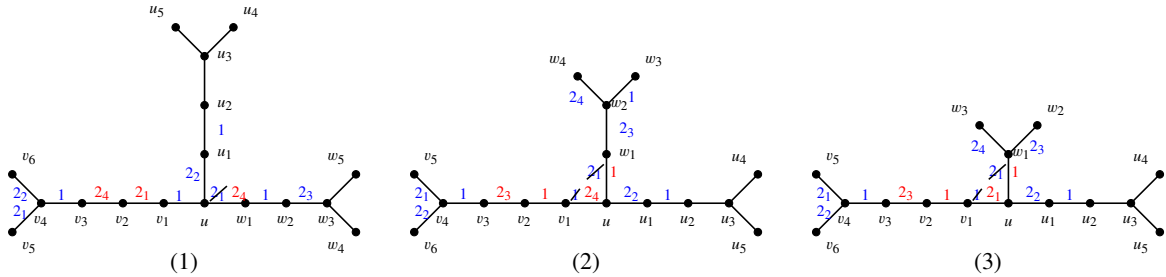

\begin{lemma}\label{no322}
If a $3$-vertex $u$ is adjacent to a $3$-thread, then it cannot be adjacent to two $2$-threads.
\end{lemma}

\begin{proof}
Suppose not, we say $u$ is adjacent to two $2$-threads $uw_1w_2w_3$ and $uu_1u_2u_3$, and a $3$-thread $uv_1v_2v_3v_4$, where $w_3, u_3$, and $v_4$ are $3$-vertices. We delete $v_2$ to obtain a subcubic planar graph $G'$ with girth at least $16$. By the minimality of $G$, $G'$ has a good coloring $f$. Let $N(w_3) = \{w_2, w_4, w_5\}$ and $N(v_4) = \{v_3, v_5, v_6\}$. By symmetry, $f(uv_1), f(v_3v_4) = 1, 1$ or $1, 2_1$ or $2_1, 1$ or $2_1, 2_1$ or $2_1, 2_2$.

In case $f(uv_1), f(v_3v_4) = 1, 1$, by Claim \ref{claim 1}, we may assume without loss of generality that $f(uw_1) = f(v_4v_5) = 2_1$ and $f(uu_1) = f(v_4v_6) = 2_2$. Furthermore, it holds that $f(u_1u_2) = 1$ and $f(w_1w_2) = 1$.

By symmetry, $f(w_2w_3)=2_2$ or $2_3$. In the former case, we can recolor $uw_1$ with $2_4$, color $v_1v_2$ with $2_1$, and color $v_2v_3$ with $2_4$ to obtain a good coloring (see Figure~\ref{lemma 4.9, 4.10 and 4.11} (1)) unless $w_2$ sees all five colors. However, we can recolor $uw_1$ and $v_2v_3$ with $2_3$ to obtain a good coloring. In the latter case, we can recolor $uw_1$ with $2_4$, color $v_1v_2$ with $2_1$, and color $v_2v_3$ with $2_4$ to obtain a good coloring unless $w_2$ sees all five colors. In this case, $\{f(w_3w_4), f(w_3w_5)\} = \{2_1,2_2\}$. However, we can recolor $uw_1$ with $1$, $w_1w_2$ with $2_3$, $w_2w_3$ with $1$, $uv_1$ with $2_4$, $v_1v_2$ with $1$ and $v_2v_3$ with $2_3$ to obtain a good coloring.

The proofs for the remaining four cases are identical to that of Lemma \ref{no333}. This is a contradiction.
\end{proof}

\begin{lemma}\label{no321}
If a $3$-vertex $u$ is adjacent to a $3$-thread and a $2$-thread, then it cannot be adjacent to a $1$-thread.
\end{lemma}
\begin{proof}
Suppose not, say $u$ is adjacent to a $3$-thread $uv_1v_2v_3v_4$, a $2$-thread $uu_1u_2u_3$, and a $1$-thread $uw_1w_2$, where $v_4, u_3$, and $w_2$ are $3$-vertices. We delete $v_2$ to obtain a subcubic planar graph $G'$ with girth at least $16$. By the minimality of $G$, $G'$ has a good coloring $f$. Let $N(w_2) = \{w_1, w_3, w_4\}, N(u_3) = \{u_2, u_4, u_5\}$, and $N(v_4) = \{v_3, v_5, v_6\}$. By symmetry, $f(uv_1), f(v_3v_4) = 1, 1$ or $1, 2_1$ or $2_1, 1$ or $2_1, 2_1$ or $2_1, 2_2$.

In case $f(uv_1), f(v_3v_4) = 1, 1$, by Claim \ref{claim 1}, we may assume without loss of generality that $f(uw_1) = f(v_4v_5) = 2_1$ and $f(uu_1) = f(v_4v_6) = 2_2$. Furthermore, it holds that $f(u_1u_2) = 1$.

If $f(w_1w_2)$ is a $2$-color, say $f(w_1w_2) = 2_3$, then by Condition 1, one of $f(w_2w_3)$ and $f(w_2w_4)$ is 1. Since $2_1 \notin \{f(w_2w_3),f(w_2w_4)\}$, we can recolor $uw_1$ with 1 and $uv_1$ with $2_4$, and color $v_1v_2$ with 1 and color $v_2v_3$ with $2_3$ to obtain a good coloring (see Figure~\ref{lemma 4.9, 4.10 and 4.11} (2)). Therefore, $f(w_1w_2) = 1$. By applying Condition 2 at $w_1$ and symmetry, we may assume $\{f(w_2w_3), f(w_2w_4)\} = \{2_2, 2_3\}$. We recolor $uw_1$ with $2_4$, and color $v_1v_2$ with $2_1$, and $v_2v_3$ with $2_4$ to obtain a good coloring. 

The proofs for the remaining four cases are identical to that of Lemma~\ref{no333}. This is a contradiction.
\end{proof}

\begin{lemma}\label{no320}
A $3$-vertex cannot be adjacent to a $3$-thread and a $2$-thread at the same time.
\end{lemma}

\begin{proof}
Suppose not, say $u$ is adjacent to a $3$-thread $uv_1v_2v_3v_4$, a $2$-thread $uu_1u_2u_3$, and another vertex $w_1$, where $v_4$, $u_3$ are $3$-vertices. By Lemmas~\ref{no330},~\ref{no322} and~\ref{no321}, $w_1$ is a $3$-vertex. We delete $v_2$ to obtain a subcubic planar graph $G'$ with girth at least $16$. By the minimality of $G$, $G'$ has a good coloring $f$. Let $N(w_1) = \{u, w_2, w_3\}$, $N(u_3) = \{u_2, u_4, u_5\}$, and $N(v_4) = \{v_3, v_5, v_6\}$. 
By symmetry, $f(uv_1), f(v_3v_4) = 1, 1$ or $1, 2_1$ or $2_1, 1$ or $2_1, 2_1$ or $2_1, 2_2$.

In case $f(uv_1), f(v_3v_4) = 1, 1$, by Claim \ref{claim 1}, we may assume without loss of generality that $f(uw_1) = f(v_4v_5) = 2_1$ and $f(uu_1) = f(v_4v_6) = 2_2$. Furthermore, it holds that $f(u_1u_2) = 1$.

We claim $1 \in \{f(w_1w_2), f(w_1w_3)\}$, since otherwise, $\{f(w_1w_2), f(w_1w_3)\} = \{2_3, 2_4\}$, we can switch the colors of $uv_1$ and $uw_1$, then color $v_1v_2$ with $1$, and $v_2v_3$ with $2_3$ to obtain a good coloring (see Figure~\ref{lemma 4.9, 4.10 and 4.11} (3)). By symmetry, we may assume $\{f(w_1w_2), f(w_1w_3)\} = \{1, 2_3\}$.

\textbf{Subcase (i):} $f(u_2u_3) = 2_1$. We claim $\{f(u_3u_4), f(u_3u_5)\} = \{2_2, 2_3\}$, since otherwise, we can recolor $uu_1$ with $2_4$, and then color $v_1v_2$ with $2_2$ and $v_2v_3$ with $2_4$ to obtain a good coloring. However, we can recolor $u_1u_2$ with $2_4$, $u_2u_3$ and $uu_1$ with $1$, and $uv_1$ with $2_2$. Subsequently, we color $v_1v_2$ with $1$ and $v_2v_3$ with $2_3$ to obtain a good coloring.

\textbf{Subcase (ii):} $f(u_2u_3) = 2_3$. We claim $\{f(u_3u_4), f(u_3u_5)\} = \{2_1, 2_2\}$, since otherwise, we can recolor $uu_1$ with $2_4$, and then color $v_1v_2$ with $2_2$ and color $v_2v_3$ with $2_4$ to obtain a good coloring. However, we can switch the colors of $u_1u_2$ and $u_2u_3$, recolor $uu_1$ with $1$ and $uv_1$ with $2_4$. Subsequently, we color $v_1v_2$ with $1$ and $v_2v_3$ with $2_3$ to obtain a good coloring.

\textbf{Subcase (iii):} $f(u_2u_3) = 2_4$. By symmetry, there are the following subcases: $f(u_3u_4), f(u_3u_5) = 1, 2_1$ or $1, 2_2$ or $1, 2_3$ or $2_1, 2_2$ or $2_2, 2_3$. 
In the first and third subcases, we can recolor $u_1u_2$ with $2_2$, $uu_1$ with $1$, and $uv_1$ with $2_4$. Then we color $v_1v_2$ with $1$ and $v_2v_3$ with $2_3$ to obtain a good coloring. 
In the second subcase, we can recolor $u_1u_2$ with $2_3$, $uu_1$ with $1$, and $uv_1$ with $2_4$. Then we color $v_1v_2$ with $1$ and $v_2v_3$ with $2_3$ to obtain a good coloring. 
In the fourth subcase, we can recolor $u_2u_3$ with $1$, $u_1u_2$ with $2_3$,  $uu_1$ with $1$, and $uv_1$ with $2_4$. Then we color $v_1v_2$ with $1$ and color $v_2v_3$ with $2_3$ to obtain a good coloring. 
In the fifth subcase, we can switch the colors of $u_1u_2$ and $u_2u_3$, and switch the colors of $uv_1$ and $uu_1$. Then we color $v_1v_2$ with $1$ and color $v_2v_3$ with $2_3$ to obtain a good coloring.

The proofs for the remaining four cases are identical to that of Lemma~\ref{no333}. This is a contradiction.
\end{proof}

Next, we establish a crucial structural property regarding the density of $3$-vertices on a facial cycle in our minimal counterexample $G$.

\begin{lemma}\label{cycle}
Let $C$ be a facial cycle of length $\ell$. There are at least $\lceil \frac{\ell}{3} \rceil$ vertices of degree 3.
\end{lemma}

\begin{proof}
Let $v_1, v_2, \ldots, v_\ell, v_1$ be the vertices of the facial cycle $C$. Let $n_2, n_3$ be the number of vertices of degree two and three in $C$. We count the number of $3$-vertices in each six-vertex path starting from $v_1, v_2, \ldots, v_\ell$ (e.g., the six-vertex path at $v_1$ is $v_1v_2v_3v_4v_5v_6$, \dots, the six-vertex path at $v_\ell$ is $v_\ell v_1 v_2 v_3 v_4 v_5$). Every $3$-vertex is counted exactly six times. Since $G$ is a minimal counterexample, by Lemmas \ref{mindegree Thm2}, \ref{no 4-chain}, \ref{no333}, \ref{no332}, \ref{no331}, \ref{no322}, \ref{no330}, \ref{no321} and \ref{no320}, it contains no 4-chain and no $3$-vertex adjacent to both a $3$-thread and a $2$-thread. Thus, every six-vertex path contains at least two $3$-vertices. Therefore,
$$ 6n_3 \ge 2\ell \quad \text{and} \quad n_3 \ge \frac{\ell}{3}. $$
Since $n_3$ is an integer, the claim is proved.
\end{proof}

We are now ready to prove Theorem \ref{Thm 2} by the discharging method. Let $V(G)$ and $F(G)$ be the collection of vertices and faces in $G$. By Euler's formula on connected planar graphs,
$$ 7 \sum_{v \in V(G)} \left(d(v) - \frac{16}{7}\right) + \sum_{f \in F(G)} (d(f) - 16) = -32. $$

We assign an initial charge of $\text{ch}(v) = 7d(v) - 16$ to each vertex $v \in V(G)$, and an initial charge of $\text{ch}(f) = d(f) - 16$ to each face $f \in F(G)$. Since there is no $1$-vertex and $G$ has girth at least $16$, each $2$-vertex has an initial charge of $-2$, each $3$-vertex has an initial charge of $5$, and each face has initial charge $d(f) - 16 \ge 0$. We redistribute the charges according to the following rules:

\textbf{Rule 1:} Every $3$-vertex sends $\frac{5}{3}$ to each of its incident faces.

\textbf{Rule 2:} Every face sends $1$ to each $2$-vertex incident to it.

Since each $3$-vertex is incident with three faces, its final charge is exactly $5 - 3 \cdot \frac{5}{3} = 0$. Each $2$-vertex receives $1$ from each incident face, and thus it has final charge $-2 + 2 \cdot 1 = 0$. At last, we show that every face has a nonnegative final charge. Let $F$ be a face with length $\ell \ge 16$. By Lemma \ref{cycle}, $F$ has at least $\lceil \frac{\ell}{3} \rceil$ vertices of degree 3 and at most $\ell - \lceil \frac{\ell}{3} \rceil$ vertices of degree 2. Therefore, $F$ has final charge at least
$$ \ell - 16 + \left\lceil \frac{\ell}{3} \right\rceil \cdot \frac{5}{3} - \left(\ell - \left\lceil \frac{\ell}{3} \right\rceil\right) \cdot 1 = \frac{8}{3}\left\lceil \frac{\ell}{3} \right\rceil - 16 \ge \frac{8}{3} \left\lceil \frac{16}{3} \right\rceil - 16 = \frac{8}{3} \cdot 6 - 16 = 0. $$
This is a contradiction, which completes the proof of Theorem \ref{Thm 2}. \hfill \qed

\section{Remarks and Future Directions}
It is worth pointing out that the structural properties derived so far are sufficient to guarantee a good coloring for any subcubic graph $G$ provided that $\text{mad}(G) < \frac{9}{4}$. We can confirm this maximum average degree condition through a standard discharging procedure. Let each vertex $v \in V(G)$ be assigned an initial charge $\text{ch}(v) = d(v) - \frac{9}{4}$. Under the assumption that $\text{mad}(G) < \frac{9}{4}$, the sum of all initial charges is strictly negative:
$$ \sum_{v \in V(G)} \text{ch}(v) = \sum_{v \in V(G)} \left(d(v) - \frac{9}{4}\right) = 2|E(G)| - \frac{9}{4}|V(G)| < 0. $$

Next, we transfer these charges based on the following rule:

\textbf{Rule 1:} A $3$-vertex transfers a charge of $\frac{1}{8}$ to every $2$-vertex it connects to through a path of $2$-vertices.

From Lemma \ref{mindegree Thm2}, we know that $G$ has no $1$-vertices. Thus, we only need to verify the final charge, denoted by $\text{ch}^*(v)$, for vertices of degree 2 and 3. According to Lemma \ref{no 4-chain}, \ref{no333}, \ref{no332}, \ref{no331}, and \ref{no322}, any $3$-vertex can reach a maximum of six $2$-vertices via paths of $2$-vertices. As a result, the final charge for an arbitrary $3$-vertex $u$ satisfies $\text{ch}^*(u) \ge 3 - 6 \cdot \frac{1}{8} - \frac{9}{4} = 0$. On the other hand, every $2$-vertex $v$ is linked to exactly two $3$-vertices through paths of $2$-vertices. Its final charge is therefore $\text{ch}^*(v) = 2 + 2 \cdot \frac{1}{8} - \frac{9}{4} = 0$.

Consequently, the total final charge must be non-negative:
$$ 0 \le \sum_{v \in V(G)} \text{ch}^*(v) = \sum_{v \in V(G)} \text{ch}(v) < 0, $$
yielding an obvious contradiction.

Recall from a standard application of Euler's formula that any planar graph with a girth of $g$ will satisfy $\text{mad}(G) < \frac{2g}{g-2}$. Our discharging procedure, however, implies that a minimal counterexample must have $\text{mad}(G) \ge \frac{9}{4}$. By combining these two bounds, we obtain $\frac{2g}{g-2} > \frac{9}{4}$, which strictly forces $g < 18$. Therefore, a minimal counterexample cannot exist when the girth is 18 or larger. This establishes that any subcubic planar graph with girth at least 18 admits a good coloring.

Gastineau and Togni~\cite{GT1} asked the open question ``is it true that all cubic graphs of girth at least $5$ are $(1,2^5)$-packing edge-colorable?''. We believe the answer is true for subcubic planar graphs. 

\begin{conj}
Every subcubic planar graph of girth at least $5$ is $(1,2^5)$-packing edge-colorable.   
\end{conj}

We end this paper by proposing another open question regarding $(1,2^4)$-packing edge-coloring.

\begin{pblm}
What is the minimum positive integer $k_2$ such that every subcubic planar graph of girth at least $k_2$ is $(1,2^4)$-packing edge-colorable? We already knew $6 \le k_2 \le 16$.  
\end{pblm}

\end{document}